\documentclass[11pt, reqno]{amsart}
\usepackage[margin=2.7cm]{geometry}
\usepackage{amsmath, amsthm, amscd, amsfonts, amssymb, graphicx, bbm, enumitem,stmaryrd}
\usepackage{subcaption}
\usepackage[colorlinks,
            urlcolor=purple]{hyperref}
\usepackage[ruled,vlined]{algorithm2e}   
\usepackage{cleveref} 
\usepackage{multirow}
\usepackage{xcolor}
\usepackage{tcolorbox}

\setlist[itemize]{leftmargin=1cm}
\setlist[enumerate]{leftmargin=1cm}

\makeatletter

\renewcommand\subsubsection{\@secnumfont}{\bfseries}%
\renewcommand\subsubsection{\@startsection{subsubsection}{3}
	\z@{.5\linespacing\@plus.7\linespacing}{-.5em}%
	{\normalfont\bfseries}}

\makeatother

\crefname{algorithm}{Algorithm}{Algorithm}
\Crefname{algorithm}{Algorithm}{Algorithm}
\crefname{definition}{Definition}{Definition}
\Crefname{definition}{Definition}{Definition}
\crefname{assumption}{Assumption}{Assumption}
\Crefname{assumption}{Assumption}{Assumption}
\crefname{section}{Section}{Section}
\Crefname{section}{Section}{Section}
\crefname{lemma}{Lemma}{Lemma}
\Crefname{lemma}{Lemma}{Lemma}
\crefname{enumi}{}{}

\newtheorem{theorem}{Theorem}[section]
\newtheorem{lemma}[theorem]{Lemma}
\newtheorem{proposition}[theorem]{Proposition}

\theoremstyle{definition}
\newtheorem{definition}[theorem]{Definition}
\newtheorem{example}[theorem]{Example}

\newtheorem{assumption}[theorem]{Assumption}

\newtheorem{heuristic}[theorem]{Heuristic}
\theoremstyle{remark}
\newtheorem{remark}[theorem]{Remark}
\numberwithin{equation}{section}

\newcommand{\bx}{\boldsymbol{x}}
\newcommand{\ba}{\boldsymbol{a}}
\newcommand{\bw}{\boldsymbol{w}}
\newcommand{\bu}{\boldsymbol{u}}
\newcommand{\bv}{\boldsymbol{v}}
\newcommand{\bz}{\boldsymbol{z}}
\newcommand{\bA}{\boldsymbol{A}}
\newcommand{\bC}{\boldsymbol{C}}

\newcommand{\cA}{\mathcal{A}}
\newcommand{\cH}{\mathcal{H}}
\newcommand{\cM}{\mathcal{M}}
\newcommand{\cC}{\mathcal{C}}

\newcommand{\cG}{\mathcal{G}}
\newcommand{\cL}{\mathcal{L}}
\newcommand{\cE}{\mathcal{E}}
\newcommand{\cN}{\mathcal{N}}
\newcommand{\N}{\mathbb{N}}
\newcommand{\R}{\mathbb{R}}

\newcommand{\real}{\mathbb{R}}
\newcommand{\one}{\mathbbm{1}}
\newcommand{\refallprop}{\eqref{prop:splitting}--\eqref{prop:averaged}}

\DeclareMathOperator{\Null}{Null}
\DeclareMathOperator{\prox}{prox}
\DeclareMathOperator{\Ran}{Ran}
\DeclareMathOperator{\zer}{Zer}
\DeclareMathOperator{\dom}{dom}
\DeclareMathOperator{\Fix}{Fix}
\DeclareMathOperator{\Tr}{Tr}
\DeclareMathOperator{\diag}{diag}

\DeclareMathOperator{\gra}{gra}
\DeclareMathOperator{\slt}{slt}

\DeclareMathOperator{\Var}{Var}
\DeclareMathOperator{\kron}{\otimes}
\DeclareMathOperator{\Id}{Id}

\usepackage[colorinlistoftodos]{todonotes}

\SetSymbolFont{stmry}{bold}{U}{stmry}{m}{n}

\begin{document}
	\centerline{}
	
	\centerline{}
	
	\title[Splitting the Forward-Backward Algorithm]{Splitting the Forward-backward Algorithm: A Full Characterization}
	
	\author[A. Åkerman \and E. Chenchene \and P. Giselsson \and E. Naldi]{A. Åkerman$^{*}$ \and  E. Chenchene$^{\dagger, c}$ \and  P. Giselsson$^*$ \and E. Naldi$^\ddagger$}

    \address{$^{*}$ Department of Automatic Control
Lund University, Lund, Sweden.}
	\email{\textcolor[rgb]{0.00,0.00,0.84}{\{anton.akerman, pontus.giselsson\}@control.lth.se}}
    
	\address{$^{\dagger}$ Faculty of Mathematics, University of Vienna, Austria.}
	\email{\textcolor[rgb]{0.00,0.00,0.84}{enis.chenchene@univie.ac.at}}
	
	\address{$^{\ddagger}$  Malga Center, DIMA, Università degli Studi di Genova, Genova, Italy.}
	\email{\textcolor[rgb]{0.00,0.00,0.84}{emanuele.naldi@edu.unige.it}}
	
	

	
	\date{\today
		\newline \indent $^{c}$ Corresponding author}
	
	\begin{abstract}
		 We study frugal splitting algorithms with minimal lifting for solving monotone inclusion problems involving sums of maximal monotone and cocoercive operators. Building on a foundational result by Ryu, we fully characterize all methods that use only individual resolvent evaluations, direct evaluations of cocoercive operators, and minimal memory resources while ensuring convergence via averaged fixed-point iterations. We show that all such methods are captured by a unified framework, which includes known schemes and enables new ones with promising features. Systematic numerical experiments lead us to propose three design heuristics to achieve excellent performances in practice, yielding significant gains over existing methods.
	\end{abstract} \maketitle

\setcounter{tocdepth}{2}

\section{Introduction}

Splitting algorithms lie at the hearth of nonsmooth optimization. They define a class of iterative method that allow to solve highly structured problems by decomposing them into simpler components, each handled through efficient and easily computable operations. In this paper, we consider monotone inclusion problems of the form:
\begin{equation}\label{eq:Inclusion}	
\text{Find} \ x \in \cH \ \text{such that:} \quad 0 \in \sum_{i=1}^nA_i(x) + \sum_{i=1}^mC_i(x),
\end{equation}
where $\cH$ is a real Hilbert space, each operator $A_i:\cH\rightarrow 2^\cH$ is maximal monotone, and $C_i: \cH\rightarrow \cH$ is $\frac{1}{\beta_i}$-cocoercive for some $\beta_i\in \real_{+}$. If the resolvent $J_{\gamma A}:=(\Id + \gamma A)^{-1}$ with $A:=A_1+\dots + A_n$ can be computed efficiently for any given step size $\gamma >0$, one can address \eqref{eq:Inclusion} through the celebrated \emph{forward-backward} algorithm \cite{lm79} that, from a starting point $x^0 \in \cH$, iterates:
\begin{equation}\label{eq:intro_forward_backward}
	x^{k+1} = J_{\gamma A}(x^k - \gamma C(x^k)), \quad \text{for all} \ k \in \N,
\end{equation}
where $C:=C_1+\dots + C_m$. The sequence generated by \eqref{eq:intro_forward_backward} is well known to converge weakly to a solution to \eqref{eq:Inclusion}, provided that one exists and $\gamma < \frac{2}{\beta}$. However, in practice the operator $J_{\gamma A}$ can rarely be computed efficiently. The classical alternative if $n=2$ and $m=0$ is the \emph{Douglas--Rachford} splitting (DRS) algorithm \cite{lm79} and if $m=1$, the so-called \emph{three-operator splitting}, or \emph{Davis--Yin} method \cite{dy17, rfp13}, that, given $z^0 \in \cH$, iterates:
\begin{equation}\label{eq:intro_david_yin}
	\left\{\begin{aligned}
		&x_1^{k+1} = J_{\gamma A_1}(z^k),\\
		&x_2^{k+1} = J_{\gamma A_2}(2x_1^{k+1} - \gamma C(x_1^{k+1}) - z^k),\\
		&z^{k+1} = z^k + (x_2^{k+1} - x_1^{k+1}),
	\end{aligned}\right. \quad \text{for all} \ k \in \N.
\end{equation}
The method reduces to DRS if $C=0$, and to \eqref{eq:intro_forward_backward} if $A_1=0$ and, in particular, is such that $\{x_1^{k+1}\}_k$ and $\{x_2^{k+1}\}_k$ converge weakly to the same solution to \eqref{eq:Inclusion} provided that $\gamma < \frac{2}{\beta}$ \cite{dy17}. 

In the general case, one possible approach to tackle \eqref{eq:Inclusion} is using a classical lifting trick \cite{pierra84}, i.e., applying \eqref{eq:intro_forward_backward} to a suitable product-space reformulation of \eqref{eq:Inclusion}, see, e.g., \cite{rfp13}. As these are formulated on product spaces, they usually feature $n$, or, in some specific instances \cite{ckch23, campoy22}, $n-1$ variables instead of only one as in \eqref{eq:intro_forward_backward} and \eqref{eq:intro_david_yin}. More recently, a number of new splitting algorithms with $n-1$ variables have been discovered, without using reformulations on product spaces but proximal-point-based designs. Notable examples include the graph-DRS algorithmic framework \cite{bcn24} for $m=0$, or the graph-forward-backward algorithm \cite{acl24} for $m=n-1$.

The need to introduce $n-1$ (or more) variables is not a coincidence. A breakthrough result by Ryu \cite{Ryu}, later extended in \cite{mt23, mbg22}, demonstrated that \emph{frugal splitting methods}---those relying solely on individual resolvent evaluations $J_{\gamma_1 A_1}, \dots, J_{\gamma_n A_n}$ with $\gamma_1, \dots, \gamma_n > 0$, forward evaluations $C_1,\dots, C_m$, and simple algebraic operations---cannot solve arbitrary instances of \eqref{eq:Inclusion} by storing fewer than $n-1$ variables between iterations. Consequently, methods achieving this bound are said to exhibit \emph{minimal lifting}. Even more striking, if $n = 2$ and $m=0$, the only frugal resolvent method with minimal lifting that is so-called \emph{unconditionally stable} (i.e., producing a weakly converging sequence to a solution for any starting point and any operator choice) is indeed DRS \cite[Corollary 1]{Ryu}. For $n > 2$, however, the landscape of such methods becomes richer, giving rise to a variety of structurally diverse algorithms that have attracted significant attention in recent years \cite{mt23, mbg22, abt23, amtt23, campoy22}. These can all be understood as fixed-point iterations with respect to \emph{averaged} operators, which  lead us to the core question of this paper:\smallskip
\begin{center}
	\emph{Can we characterize all averaged frugal splitting methods with minimal lifting to solve \eqref{eq:Inclusion}?}
\end{center}\smallskip
To be more precise, we consider fixed-point iterations with respect to $T: \cH^{n-1}\rightarrow \cH^{n-1}$ to solve arbitrary instances of \eqref{eq:Inclusion}, with the following properties:
    \begin{enumerate}[label=(P\arabic*),ref=P\arabic*]
    \item \label{prop:splitting}
    \emph{Splitting}: Each evaluation of $T$ is constructed solely from resolvent evaluations of \(A_i\), direct evaluations of \(C_i\), and arbitrary linear combinations of their inputs and outputs. 
    \item \label{prop:frugal} \emph{Frugality}: Each operator is evaluated exactly once per evaluation of $T$, either directly for \(C_i\) or through a resolvent for \(A_i\).
    \item \label{prop:minlift} \emph{Minimal lifting}: The method only needs to store $n-1$ variables between iterations. This property is automatically satisfied for $T$ defined on $\cH^{n-1}$.
    \item \label{prop:FPE} \emph{Fixed-point encoding}: Fixed points of $T$ correspond to solutions of \eqref{eq:Inclusion} and vice versa, as made precise in Section~\ref{sec:Conditions_for_fixed-point-encoding}.  \item \label{prop:averaged} \emph{Averaged nonexpansive}: The operator $T$ is $\theta$-averaged on $\cH^{n-1}$, $\theta \in (0, 1)$, equipped with the product norm.
    \end{enumerate}

    In this paper, we give a positive answer to the question showing that the entire class of methods satisfying properties \refallprop\ is given by Algorithm~\ref{alg:simple_splitting_FPE_introduction}, meaning that all methods covered by Algorithm~\ref{alg:simple_splitting_FPE_introduction} satisfy \refallprop\ and that no other methods satisfy all these requirements. This represents the first, yet partial, answer to a challenging open problem proposed by Ryu in \cite{Ryu}, on the characterization of all \emph{unconditionally stable} methods.  
    
    Although the central question is of pure intellectual interest, its resolution comes with significant practical implications. Indeed, while Algorithm~\ref{alg:simple_splitting_FPE_introduction} can be shown to encompass several methods in the literature, it also allows us to devise new ones with desirable numerical properties. In particular, in Algorithm~\ref{alg:distributed_algorithm}, we propose a special case of Algorithm~\ref{alg:simple_splitting_FPE_introduction} that can easily handle data heterogeneity and admits fully distributed implementations on general networks, without requiring knowledge of a global cocoercivity constant. We further conduct a systematic numerical study examining the influence of the four parametrizing matrices $M$, $P$ and $H, K$ and propose three heuristics to achieve excellent performance in practice. The resulting method enhanced by our heuristics exhibits unparalleled performance, significantly outperforming all existing variants.
    
    \begin{algorithm}[t]
    	\caption{The complete class of $\theta$-averaged frugal splitting methods with minimal lifting to solve \eqref{eq:Inclusion}.}
    	\label{alg:simple_splitting_FPE_introduction}
    	\textbf{Pick:} A relaxation parameter $\theta \in (0, 1)$, and:
    	\begin{enumerate}
    		\item $M\in \R^{n\times (n-1)}$ full rank, with $\Ran (\one) = \Null( M^T)$
    		\item $P \in \R^{n\times (n-1)}$ with $ \Ran(\one) \subset \Null (P^T)$
    		\item $H, K^T \in \R^{n\times m}$ causal, according to Definition~\ref{def:causal_pair}, and such that $H^T\one=K\one=\one$ 
    	\end{enumerate}
    	\textbf{Let:} $S := MM^T + PP^T + \frac{1}{2}(H - K^T)\diag(\beta)(H^T-K)$ and $\gamma := 2\diag(S)^{-1}$ \\
    	\textbf{Input:} {$\bz^0=(z^0_1, \dots, z^0_{n-1})\in \cH^{n-1}$}\\
    	\For {$k = 0, 1, 2, \ldots$}{
    		\For {$i = 1, 2, \dots, n$}{
    			$\displaystyle x^{k+1}_i = J_{ \gamma_i  A_i}\bigg(- \gamma_i \sum_{h=1}^{i-1} S_{ih} x_h^{k+1} -  \gamma_i \sum_{j=1}^{m}H_{ij} C_j\bigg(\sum_{h = 1}^{i-1}K_{jh}x^{k+1}_h\bigg) +  \gamma_i \sum_{j=1}^{n-1}M_{ij}z^k_j \bigg)$
    		}
    		\For {$i = 1, 2, \dots, n-1$}{
    			$\displaystyle z^{k+1}_i = z^k_{i}-\theta \sum_{i=1}^n M_{ij}x^{k+1}_i$
    		}
    	}
    \end{algorithm}

	\subsection{Organization of the paper} The paper is organized as follows. In Section~\ref{sec:Algorithm Model}, we postulate our algorithm model that concretizes our \emph{splitting} definition \eqref{prop:splitting} as well as the notion of \emph{frugality} \eqref{prop:frugal}. A dimensionality constraint on the involved matrices restricts the class of algorithms to \emph{minimal lifting} methods \eqref{prop:minlift}. In Section~\ref{sec:Conditions_for_fixed-point-encoding}, we make restrictions to this model to exactly capture all algorithms for solving \eqref{eq:Inclusion} that in addition are \emph{fixed-point encodings} according to Definition~\ref{def:FPE}. In Section~\ref{sec:conditions_for_nonexpansivness}, we complete our analysis by deriving necessary and sufficient conditions for nonexpansiveness. In Section~\ref{sec:Problem_formulation_and_parameterization}, we present a reformulation of the class of algorithms leading to Algorithm~\ref{alg:simple_splitting_FPE_introduction}, along with existing and new special cases. We conclude in Section~\ref{sec:numerics} with heuristics for selecting the parameterizing matrices, a systematic numerical study and numerical comparisons with existing methods.
	
    \subsection{Contributions} Let us highlight the main contributions of this paper:
    \begin{itemize}
    	\item We show that if $m=0$, Algorithm~\ref{alg:simple_splitting_FPE_introduction} parametrizes all $\theta$-averaged frugal resolvent splitting methods with minimal lifting, cf.~Theorem~\ref{thm:frugal_resolvent_iff}.  
    	\item We show that if $m>0$ and $\dim(\cH)\geq 2n+m-1$, Algorithm~\ref{alg:simple_splitting_FPE_introduction} parametrizes all $\theta$-averaged frugal splitting methods with minimal lifting, cf.~Theorem~\ref{thm:frugal_splitting_iff}.  
    	\item We propose in Algorithm~\ref{alg:distributed_algorithm} a special case of Algorithm~\ref{alg:simple_splitting_FPE_introduction} that can handle efficiently different $\beta_i$s, and admits distributed implementations on general networks without requiring the knowledge of a global cocoercivity constants.
    	\item We propose in Section~\ref{sec:numerics} three heuristic choices of $M$, $P$ and $H, K$ that make Algorithm~\ref{alg:simple_splitting_FPE_introduction} particularly efficient compared to other existing special cases.
    \end{itemize}

\subsection{Notation}

Let $\N$ denote the set of nonnegative integers, $\R$ the set of real numbers, and $\R_+$, $\R_{++}$ the set of nonnegative and positive real numbers, respectively. Let $\llbracket i,j\rrbracket := \{i,i+1,\ldots,j\}$ with the convention that $\llbracket i,j\rrbracket=\emptyset$ if $j<i$, $\kron$ denote the tensor product, $\cH$ denote a real Hilbert space with identity operator denoted by $\Id$. For a matrix $M\in \R^{n\times m}$ we denote by $M_{ij}$ its $(i, j)$ component. We denote by $\slt( M)\in \R^{n \times m}$ the strictly lower triangular matrix extracted from $M$, and by $\diag(M) \in \R^{n}$ the vector extracted from the main diagonal of $M$. With a mild overload of notation, if $v \in \R^n$, $\diag(v) \in \R^{n\times n}$ denotes the diagonal matrix with the diagonal being $v$.

For a multivalued map $A: \cH\to 2^\cH$, its \emph{graph} is the set $\gra( A) := \{(x, a) \in \cH^2\ : \ a \in A(x)\}$, while its \emph{domain} is $\dom (A ):= \{x \in \cH \ : \ A(x) \neq \emptyset\}$. An operator $A:\cH\to 2^\cH$ is \emph{monotone} if 
\begin{equation}
    \langle a - a', x - x'\rangle \geq 0, \quad \text{for all} \ (x, a),  \ (x', a') \in \gra (A),
\end{equation}
and is \emph{maximal montone} if its graph is maximal among the class of monotone operators. The \emph{inverse} of $A$ is the unique operator $A^{-1}$ such that $\gra (A^{-1})=\{(a, x) \ : \ (x, a) \in \gra (A)\}$ and the \emph{resolvent} of $A$ is the operator $J_A:=(\Id + A)^{-1}$, which is single-valued if $A$ is monotone, see \cite[Proposition 23.8]{BCombettes}. By Minty's theorem \cite[Theorem 21.1]{BCombettes} a monotone operator is maximal if and only if its resolvent $J_A$ has full domain. In this case, it defines a map from $\cH$ to itself. In such cases, we identify multivalued operators with functions on $\cH$. A monotone operator $C: \cH \to 2^\cH$ is $\frac{1}{\beta}$-\emph{cocoercive} if it is single-valued, has full domain, and satisfies:
\begin{equation}
    \beta \langle C(x) - C(x'), x-x'\rangle \geq \|C(x) - C(x')\|^2, \quad \text{for all} \ x, x'\in \cH, 
\end{equation}
in particular, it is $\beta$-Lipschitz continuous and if $\beta=0$ it is necessarily constant. Eventually, we denote by $\cA_n$ and by $\cC_m$ the families of $n$ and $m$ tuples of maximal monotone and cocoercive operators, respectively.

We identify matrices $M \in \mathbb{R}^{n \times m}$ with block-operators $(M\kron\Id):\cH^m\to\cH^n$, when clear by context. In particular, for any $z \in \cH^m$, we abbreviate $(M \kron \Id)z$ by $Mz$. Similarly, if $\bC: \cH^n \to \cH^n$ and $\bC': \cH^m \to \cH^m$ are operators, we interpret the products $\bC M$ and $M\bC'$ as $\bC \circ (M \kron \Id)$ and $(M \kron \Id) \circ \bC'$, respectively. Further, we denote by $\mathbb{S}_+^n$ and $\mathbb{S}_{++}^n$ the spaces of positive semidefinite and positive definite matrices in $\R^{n\times n}$. Given Hilbert spaces $\cH_1,\dots, \cH_n$ and $A_i: \mathcal{H}_i\to 2^\mathcal{H}_i$, we denote by $\bA(\bx):=\diag (A_1,\dots, A_n)$ the \emph{diagonal} operator defined by $\bA(\bx):= A_1(x_1) \times A_2(x_2)\times \cdots \times A_n(x_n)$, for all $\bx=(x_1, \dots, x_n) \in \cH_1\times \dots \times \cH_n$.

\section{Algorithmic model}\label{sec:Algorithm Model}
   
To capture the entire class of methods with properties \eqref{prop:splitting} to \eqref{prop:averaged}, we start out by defining the entire class of \emph{frugal splitting methods with minimal lifting} satisfying \eqref{prop:splitting}--\eqref{prop:minlift}. 
\begin{definition}[Frugal splitting operators with minimal lifting]\label{def:frugal_minlift_splitting} Let $\Gamma:=\diag(\gamma) \in \R^{n\times n}$ with $\gamma \in \R^n_{++}$, $L\in \R^{n\times n}$, $V^T, M\in \R^{n\times (n-1)}$, $U\in \R^{(n-1)\times (n-1)}$, $H^T,K\in \real^{m\times n}$,  $Z^{(1)}, \in \real^{m\times (n-1)}$, $Z^{(2)}\in \real^{m\times m}$, and $Z^{(3)}\in \real^{(n-1)\times m}$. A \emph{frugal splitting operator with minimal lifting} is an operator $T: \cH^{n-1}\to \cH^{n-1}$ defined by $\bz^+:= T(\bz)$ with:
\begin{equation}\label{eq:abstract_fs_block}
   \left\{
	\begin{aligned}
	 &\bx^{+} = J_{\Gamma \bA}(\Gamma L \bx^+ - \Gamma H\bu^+ + \Gamma M \bz ),\\
     &\bu^+ = \bC(K\bx^+ + Z^{(2)}\bu^+ + Z^{(1)}\bz),\\
     & \bz^{+} = U  \bz - V  \bx^{+} - Z^{(3)}\bu^+,
    \end{aligned}
    \right. 
\end{equation}
and such that $\bz^+$ can be computed only employing additions and scalar multiplications in $\cH$, and exactly one evaluation of $J_{\gamma_i A_i}$ and $C_j$ for each $i \in \llbracket 1, n\rrbracket$ and $j \in \llbracket 1, m\rrbracket$, and for each $A\in \cA_n$, $C\in \cC_m$ and any real Hilbert space $\cH$. Given $T$ as in \eqref{eq:abstract_fs_block}, a \emph{frugal splitting method with minimal lifting} is a fixed-point iteration defined by
    \begin{equation}
    	\bz^{k+1} = T(\bz^k), \quad \text{for all} \ k \in \N, \ \bz^0 \in \cH^{n-1}.
    \end{equation}
\end{definition}

Observe that each resolvent evaluation automatically produces elements $a_i^+ \in A_i(x_i^+)$ for all $i\in \llbracket1, n\rrbracket$. The evaluation of $T$ as in \eqref{eq:abstract_fs_block} at some $\bz\in \cH^{n-1}$ can be thus characterized as the \emph{unique} solution to the following system of nonlinear equations:
\begin{equation}\label{eq:abstract_fs_inclusion}
\begin{aligned}
    &\text{Find $\bz^+ \in \cH^{n-1}$, $\bx^+\in \cH^{n}$, $\bu^+ \in \cH^m$, $\ba^+\in \bA(\bx^+)$ such that:}\\
    &\quad \left\{
    \begin{aligned}
	 &(\Gamma^{-1}- L)\bx^{+} + \ba^+ = M \bz - H \bu^+,\\
     &\bu^+ = \bC(K\bx^+ + Z^{(2)}\bu^+ + Z^{(1)}\bz),\\
     & \bz^{+} = U  \bz - V  \bx^{+} - Z^{(3)}\bu^+.
    \end{aligned}
    \right. 
\end{aligned}
\end{equation}
In the following, we say that $\bz^+ \in \cH^{n-1}$, $\bx^+\in \cH^{n}$, $\bu^+ \in \cH^m$, $\ba^+\in \bA(\bx^+)$ \emph{solve} \eqref{eq:abstract_fs_inclusion} at $\bz$. Alternatively, we say that $\bx^+$, $\ba^+$ and $\bu^+$ are \emph{generated} by the evaluation of $T(\bz)$.

The operator in \eqref{eq:abstract_fs_block} is \emph{frugal} \eqref{prop:frugal} since each resolvent of $A_i$ and each $C_i$ can be evaluated exactly once per iteration and it has \emph{minimal lifting} \eqref{prop:minlift} since the corresponding fixed-point iteration lies in $\cH^{n-1}$. Being able to explicitly evaluate $T$ immediately induces a specific structure on $L$, and the matrices $H$, $K$ and $Z^{(2)}$, a property that we refer to as \emph{causality}. Obviously, up to a permutation, $L$ and $Z^{(2)}$ should be strictly lower triangular. For  $H$ and $K$ the situation is more subtle---we discuss their structural properties in Section~\ref{sec:causality}. Additionally, our algorithm model includes methods that may or may not satisfy \eqref{prop:FPE} and \eqref{prop:averaged}, i.e., that may or may not have the \emph{fixed-point encoding property}, and be \emph{averaged nonexpansive}. As we proceed, we will provide necessary and sufficient conditions for each of these properties, in terms of structural constraints of the involved matrices. 

    \subsection{Causality}\label{sec:causality}
    
    In order for an operator $T$ defined as in \eqref{eq:abstract_fs_block} to be well-defined, specifically to only requiring single evaluations of $J_{\gamma_i A_i}$ and $C_j$ for each $A \in \cA_n$ and $C \in \cC_m$, the matrices $L$, $H$, $K$ and $Z^{(2)}$ need to meet certain structural properties.
    
    First, let us observe that any frugal splitting method automatically defines a partial order (denoted by $\hookrightarrow$) on the evaluations of operators $A_1,  \dots, A_n$, $C_1, \dots, C_m$. As this ordering is implicit in Definition~\ref{def:frugal_minlift_splitting}, we will assume it is explicitly given, and we require further structure on the matrices of \eqref{eq:abstract_fs_block} in order to not violate it. Additionally, note that, up to relabeling, we shall also assume that $A_i \hookrightarrow A_j$ and, similarly, $C_i \hookrightarrow C_j$, for each $i<j$. To comply with \eqref{eq:abstract_fs_block}, we necessarily have $A_1\hookrightarrow C_j$ and $C_j \hookrightarrow A_n$ for all $j \in \llbracket 1, m\rrbracket$.

    Second, we find that the ordering $\hookrightarrow$ is uniquely characterized by a vector $F \in \N^n$ with $F_1=0$, $F_n=m$ and $F_i \leq F_{i+1}$ for each $i \in \llbracket 1, n-1\rrbracket$, a notion that we refer to as a \emph{$(m, n)$-nondecreasing vector} (see Proposition~\ref{prop:Causal_F}). In our algorithms, we will interpret $F_i$ as the number of forward operators that are evaluated before applying the $i$th backward operator in each iteration. In \eqref{eq:abstract_fs_block}, backward steps are evaluated both first and last, which forces $F_1 = 0$ and $F_n = m$.
    
    With this notation, we can show the following, whose proof is postponed to the appendix.
    \begin{proposition}
        \label{prop:causal_iff}
        Each frugal splitting operator with minimal lifting as parameterized by \eqref{eq:abstract_fs_block} is well-defined if and only if $L$ and $Z^{(2)}$ are strictly lower triangular and $H,K^T$ is a causal pair of matrices according to Definition~\ref{def:causal_pair}.
    \end{proposition}

\subsection{Equivalence of algorithms}
\label{sec:NE-norm_and_equivalence}

Following \cite{Ryu}, we now introduce a notion of \emph{equivalence} between frugal splitting operators that immediately reduces the degrees of freedom in \eqref{eq:abstract_fs_block}.

\begin{definition}[Equivalent frugal splitting algorithms]\label{def:equivalence}
	Two frugal splitting methods with minimal lifting are \emph{equivalent} if the corresponding frugal splitting operators $T$ and $\bar T$ satisfy at least one of the following:
	\begin{enumerate}[label=(\roman*)]
    \item\label{item:equivalence_linear_transformation} There exists an invertible matrix $\mathfrak{E} \in \R^{(n-1)\times (n-1)}$ such that $\bar T =  \mathfrak{E}^{-1}T(\mathfrak{E})$ for any choice of $A\in \cA_n$ and $C\in \cC_m$ and any real Hilbert space $\cH$,
	\item $\bar T$ can be obtained from $T$ by swapping the order of $A_1, \dots, A_n, C_1, \dots, C_m$, or by scaling these monotone operators by some $\tau > 0$.
	\end{enumerate}
\end{definition}

Frugal splitting methods with respect to equivalent operators generate sequences that indeed coincide modulo an isomorphism of $\cH^{n-1}$, or can be obtained with a different step size $\tau>0$ or by swapping the order of the operators. Since in Section~\ref{sec:causality} we fixed the operators and their order of evaluation, in what remains, two equivalent methods are always understood in the sense of Item~\ref{item:equivalence_linear_transformation}. In the following result, we show that this is tightly related to a change of metric in $\cH^{n-1}$. We will use this result in Section~\ref{sec:conditions_for_nonexpansivness} to fix a-priori the canonical norm.

\begin{proposition}\label{prop:equivalence}
A frugal splitting operator $T$ with minimal lifting according to Definition~\ref{def:frugal_minlift_splitting} is nonexpansive in the norm $\|\cdot \|_Q$ for $Q \in \mathbb{S}_{++}^{n-1}$ if and only if there exists an equivalent operator $\bar T$ that is nonexpansive w.r.t.~the canonical norm.
\end{proposition}
\begin{proof}
	Let $\mathfrak{E} := Q^{-\frac{1}{2}}$, $T$ be nonexpansive in the norm $\|\cdot\|_Q$, and define $\bar T := \mathfrak{E}^{-1} T(\mathfrak{E})$. Then $\bar T$ is nonexpansive with respect to the canonical norm since for all $\bz, \bz' \in \cH^{n-1}$
	\begin{equation}
		\begin{aligned}
			\|\bar T (\bz)- \bar T(\bz')\| &= \|Q^{\frac{1}{2}}T(Q^{-\frac{1}{2}}\bz)-Q^{\frac{1}{2}}T(Q^{-\frac{1}{2}}\bz')\| = \|T(Q^{-\frac{1}{2}}\bz)-T(Q^{-\frac{1}{2}}\bz')\|_Q \\
			& \leq \|Q^{-\frac{1}{2}}\bz-Q^{-\frac{1}{2}}\bz'\|_Q = \|\bz - \bz'\|.
		\end{aligned}
	\end{equation}
	Additionally, $\bar T$ is an instance of \eqref{eq:abstract_fs_block} with $\bar U := \mathfrak{E}^{-1}U \mathfrak{E}$, $\bar V:= \mathfrak{E}^{-1}V$, $\bar Z^{(3)} := \mathfrak{E}^{-1} Z^{(3)}$, $\bar Z^{(1)}:= Z^{(1)} \mathfrak{E}$, and $\bar M:= M \mathfrak{E}$. Therefore, $T$ and $\bar T$ are equivalent according to Definition~\ref{def:equivalence}. The converse is similar.
\end{proof}

\section{Necessary and sufficient conditions for fixed-point encoding}\label{sec:Conditions_for_fixed-point-encoding}

Not all minimal lifting frugal splitting methods, as parameterized by~\eqref{eq:abstract_fs_block}, are useful for solving the inclusion problem \eqref{eq:Inclusion}. Consider for example the algorithm parameterized by $U = I$, $V = 0$, and $Z^{(3)}=0$, which gives the fixed-point iteration $\bz^{k+1} = \bz^k$. This fixed-point iteration does not depend on the operators defining the inclusion problem, and fixed points of this fixed-point iteration give no meaningful information about solutions of \eqref{eq:Inclusion}. We are interested in the exact subset of algorithms of the form \eqref{eq:abstract_fs_block} that have the fixed-point encoding property, meaning solutions to \eqref{eq:Inclusion} can easily be extracted from a fixed point of \eqref{eq:abstract_fs_block} and that the zero set of \eqref{eq:Inclusion} is nonempty if and only if the fixed-point set of \eqref{eq:abstract_fs_block} is. In the following definition, a simple adaptation of \cite{mt23, Ryu}, we concretize the fixed-point encoding property outlined in~\eqref{prop:FPE}.

\begin{definition}[Fixed-point encoding]\label{def:FPE}
	A frugal splitting operator $T$ according to Definition~\ref{def:frugal_minlift_splitting} is a \emph{fixed-point encoding} if there exist $s_z \in \R^{n-1}$, $s_x \in \R^{n}$, and $s_u \in \R^{m}$ such that the map $\mathfrak{s}:\cH^{n-1}\to \cH$ defined by
    \begin{equation}
        \mathfrak{s}(\bz):=s_z^T \bz + s_x^T \bx^+  + s_u^T \bu^+, \quad \text{for all} \ \bz \in \cH^{n-1},
    \end{equation}
    where $\bx^+, \bu^+$ are obtained through evaluation of $T(\bz)$, satisfies, for all $A \in \mathcal{A}_n$, $C \in \cC_m$ and real Hilbert space $\cH$,
	\begin{enumerate}[label=(\roman*)]
		\item\label{item:zero_empty_iff_fix_empty} $\zer\left(A_1 + \dots +A_n + C_1 + \dots + C_m\right) \neq \emptyset \iff \Fix(T) \neq \emptyset$,
		\item\label{item:fpe_solution_map_solves_pb} $\bz^\star \in \Fix(T) \implies \mathfrak{s}(\bz^\star) \in \zer\left(A_1 + \dots +A_n + C_1 + \dots + C_m\right)$.
	\end{enumerate}
    In this case, $\mathfrak{s}$ is called \emph{solution map}.
\end{definition}

Indeed, adapting \cite[Proposition 1]{mt23}, we can immediately simplify the setting, since for any fixed-point encoding, the solution map is necessarily of a specific form:
\begin{proposition}\label{prop:solution_map}
Let $T$ be a frugal splitting operator with minimal lifting according to Definition~\ref{def:frugal_minlift_splitting}. Suppose that $T$ is a fixed-point encoding according to Definition~\ref{def:FPE} and let $\mathfrak{s} : \cH^{n-1} \to \cH$ be the corresponding solution map. Then, for all $\bz \in \Fix(T)$ 
\begin{enumerate}[label=(\roman*)]
\item\label{item:prop_solution_map_sum} $a_1^+ + \dots + a_n^+ + u_1^+ + \dots + u_n^+ =0$,
\item\label{item:prop_solution_map_characterization} $\mathfrak{s}(\bz) = x_1^+ = \dots =x_n^+$,
\end{enumerate}
where $\bx^+\in \cH^n$, $\ba^+\in \cH^n$ and $\bu^+\in \cH^m$ are obtained through evaluation of $T(\bz)$.
\end{proposition}
\begin{proof}
Fix $\cH$, $A\in \cA_n$ and $C \in \cC_m$ with $\zer(A_1+\dots +A_n + C_1 + \dots + C_m) \neq \emptyset$. Let $\bz \in \Fix(T)$, and $\bx^+, \ba^+, \bu^+$ be the elements obtained by evaluation of $T$, i.e., those solving \eqref{eq:abstract_fs_inclusion}. Consider now for $j \in \llbracket0, n\rrbracket$ the tuples of operators $A^{(j)}\in \cA_n$ and $C^{(j)}\in \cC_m$ defined as follows. Set $C_i^{(j)}(x) = u_i^+$ for all $i \in \llbracket1, m\rrbracket$ and $j \in \llbracket0, n\rrbracket$. Regarding $A^{(j)}$, set, for all $x \in \cH$ and all $i \in \llbracket1, n\rrbracket$,
\begin{equation}
    A^{(0)}_i(x) := a_i^+, \quad \text{and for all $j \in \llbracket1, n\rrbracket$} \quad A^{(j)}_i(x):= \begin{cases}
        a_i^+ & \text{if $i \neq j$},\\
        a_i^+ + x - x_i^+ & \text{if $j=i$}.
    \end{cases}
\end{equation}
Let $T^{(j)}$ be the operator with respect to $A^{(j)}$ and $C^{(j)}$. By construction, $\bz \in \Fix(T^{(j)})$ for all $j \in \llbracket 0, n\rrbracket$ since $\bz, \bx^+, \bu^+$ and $\ba^+$ still solve \eqref{eq:abstract_fs_inclusion} at $\bz$ with $\bA$ replaced by $\bA^{(j)}:=\diag(A^{(j)})$, or, in other words, $\bx^{+}$, $\ba^{+}$ and $\bu^{+}$ obtained via evaluation of $T^{(j)}(\bz)$ all coincide for each $j\in \llbracket0, n\rrbracket$.

Consider the case $j=0$. By Item \ref{item:zero_empty_iff_fix_empty} in Definition \ref{def:FPE}, as $\Fix(T^{(0)}) \neq \emptyset$ we must have $\sum_{i=1}^n a_i^+ + \sum_{i=1}^m u_i^+ = 0$, which shows Item \ref{item:prop_solution_map_sum}. Consider now the case $j \in \llbracket 1, n\rrbracket$. From Item \ref{item:fpe_solution_map_solves_pb} in Definition \ref{def:FPE}, $x^\star:=\mathfrak{s}(\bz)=s_z^T \bz + s_x^T \bx^+  + s_u^T \bu^+$ must solve for all $j \in \llbracket1, n\rrbracket$
\begin{equation}
    0 \in \sum_{i=1}^n A^{(j)}_i(x^\star)+ \sum_{i=1}^m C^{(j)}_i(x^\star) = \sum_{i=1}^na_i^+ + \sum_{i=1}^m u_i^+ + (x^\star - x_j^+) = x^\star -x_j^+, 
\end{equation}
which shows Item \ref{item:prop_solution_map_characterization} and concludes the proof.
\end{proof} 

Proposition \ref{prop:solution_map} shows that if $\bz \in \Fix(T)$ then $x^+_1 = \dots = x_n^+$ and $x^\star:=x_1^+$ solves \eqref{eq:Inclusion}. Therefore, we shall set aside the notion of solution map, as it is from now on implicitly encoded in the operator itself, without requiring any further operation. 

\subsection{The case of pure resolvent splitting}

To arrive at the necessary and sufficient conditions for \eqref{eq:abstract_fs_block} to have the fixed-point encoding property, we start by providing necessary conditions for fixed-point encoding of pure \emph{resolvent} splitting methods with minimal lifting, i.e., when $m = 0$. In this case, a frugal splitting operator $T$ \eqref{eq:abstract_fs_block} is referred to as a \emph{frugal resolvent splitting operator} and    reduces to:
    \begin{equation}
        \label{eq:Block_Resolvent}
        \bz^+:= T(\bz), \quad \text{with} \quad \left\{
        \begin{aligned}   
        &\bx^+ = J_{\Gamma\bA}(\Gamma L\bx^+ +\Gamma M\bz ),\\
        &\bz^{+} = U\bz -V\bx^+.
        \end{aligned}\right.
    \end{equation}

\begin{lemma}
    \label{lemma:FPE_resolvent_necessary}
    Let $T$ be a frugal resolvent splitting operator according to \eqref{eq:Block_Resolvent} with minimal lifting. Suppose that $T$ is a fixed-point encoding. Then, the following hold:
\begin{enumerate}[label=(\roman*)]
    \item \label{lemma:FPE_resolvent_necessary:V}$\Null(V) = \Ran(\one)$,
    \item \label{lemma:FPE_resolvent_necessary:U}$U = I$,
    \item \label{lemma:FPE_resolvent_necessary:gamma}$\one^T\left(\Gamma^{-1}-L\right)\one = 0$,
    \item \label{lemma:FPE_resolvent_necessary:M}$\Null(M^T) = \Ran(\one)$.
\end{enumerate}
\end{lemma}
\begin{proof}
In this proof, we argue with $\cH=\R$.

Item~\ref{lemma:FPE_resolvent_necessary:V}:           Since $V\in\R^{(n-1)\times n}$, $\Null(V)\neq \{0\}$. Let $\bx=(x_1,\ldots,x_n)\in \Null(V)$ and take
          \begin{equation}
          	A_i := \partial\iota_{\{x_i\}}, \quad \text{for} \ i \in \llbracket1, n\rrbracket.
          \end{equation}
Observe that $\bz=0$ is a fixed point of $T$. Since $\Fix(T)\neq \emptyset$,  we get $\zer(\sum_{i=1}^nA_i)\neq \emptyset$, and thus $\bx \in \Ran(\one)$. Therefore, $\Null(V)=\Ran(\one)$.

Item~\ref{lemma:FPE_resolvent_necessary:U}: Assume by contradiction that $U \neq I$. Pick $\bv:=(I-U)\bz^\star \neq 0$ and $\bx^\star \notin \Null( V)$ such that $- V \bx^\star = \bv$. This is indeed possible since $V$ is full rank by \ref{lemma:FPE_resolvent_necessary:V}. Consider now $A_i = \partial\iota_{\{x^\star_i\}}$ for $i \in \llbracket 1,n \rrbracket$. Then, $\bz^\star\in\Fix(T)$ but $\zer\left(\sum_{i=1}^n A_i\right) = \emptyset$, which yields the desired contradiction.
            
	Item~\ref{lemma:FPE_resolvent_necessary:gamma}: Let us now consider constant operators defined by:
      \begin{equation}
      	A_i(x) = v^\star_i, \quad \text{for all} \ i \in \llbracket 1, n \rrbracket, \quad \text{with} \quad \bv^\star = (v^\star_1, \dots, v^\star_n) = -(\Gamma^{-1}-L)\one \in \real^n. 
      \end{equation}
      Then $\bz^\star := 0 \in \Fix(T)$ since, with $\bx^\star := \one$, we get by conditions \cref{lemma:FPE_resolvent_necessary:V,lemma:FPE_resolvent_necessary:U}
that:
      \begin{equation}
          \begin{aligned}
              0&=\bz^\star = \bz^\star - V\bx^\star = 0-0 = 0,\\
              0&=M\bz^\star=(\Gamma^{-1}-L+\bA)(\bx^\star) = (\Gamma^{-1}-L-(\Gamma^{-1}-L))\one=0.
          \end{aligned}
      \end{equation}
      Thus, $\zer(\sum_{i=1}^n A_i)\neq \emptyset$, i.e., $\one^T(\Gamma^{-1}-L)\one = 0$.
      
  Item~\ref{lemma:FPE_resolvent_necessary:M}:
  Given arbitrary $\bv^\star = (v_1^\star, \dots, v_n^\star)\in \R^n$ such that $\one^T \bv^\star = 0$ define the operators as:
  \begin{equation}
	A_i(x) =x + v_i^\star,  \quad  \text{for} \ i \in \llbracket 1, n\rrbracket \text{ and each } x\in\R.
  \end{equation}
   Clearly, for any such choice of $\bv^\star$, $\zer (\sum_{i=1}^n A_i) = \{0\}\neq\emptyset$, implying through the fixed-point encoding assumption that the fixed-point set is nonempty. By Proposition~\ref{prop:solution_map}, we know that $\bx^\star=0$ and by \eqref{eq:abstract_fs_inclusion}, we know that $\bz^\star$ is a fixed point if and only if it satisfies: 
    \begin{equation}\label{eq:Inclusion_FPE_M}
   M\bz^\star=(\Gamma^{-1}-L+\bA)(\bx^\star) = 0 + \bv^\star,
   \end{equation}

   As $\bv^\star\in\Ran(\one)^\perp$ is an arbitrary point in $\Ran(\one)^\perp$, and $\Fix(T)\neq \emptyset$ by fixed-point encoding, $\Ran(M)\supseteq \Ran(\one)^\perp$, which is equivalent to $\Null(M^T)\subseteq \Ran(\one)$. Since $M^T\in \R^{(n-1)\times n}$ has a non-trivial null space, we conclude that $\Null(M^T)=\Ran(\one)$.
  \end{proof}

\begin{remark}
	Observe that the counterexample in Item~\ref{lemma:FPE_resolvent_necessary:M} and specifically \eqref{eq:Inclusion_FPE_M} can be used to show that if $T$ is a fixed-point encoding, it cannot be defined on a $d$-fold product space with $d<n-1$, providing another explicit proof of \cite[Theorem 1]{mt23}.
\end{remark}

\subsection{The general case}

We now deal with the general case with $m> 0$. We first note that a frugal splitting method with minimal lifting, for which $\bC = 0$, simplifies to a pure resolvent splitting, also with minimal lifting, so the conditions of Proposition~\ref{lemma:FPE_resolvent_necessary} 
are necessary also for fixed-point encoding frugal splitting methods.
\begin{lemma}
    \label{lemma:FPE_splitting_necessary}
    Let $T$ be a frugal splitting operator with minimal lifting according to \eqref{eq:abstract_fs_block} and satisfying the fixed-point encoding property. Then, all the conditions of \cref{lemma:FPE_resolvent_necessary} hold and additionally:
    \begin{enumerate}[label=(\roman*)]
    \item  \label{lemma:FPE_splitting_necessary:HK}$H^T\one = K\one = \one$,
    \item   \label{lemma:FPE_splitting_necessary:Z}$Z^{(1)}$, $Z^{(2)}$ and $Z^{(3)}$ are zero.
\end{enumerate}
\end{lemma}
\begin{proof} Let $C_i= 0$ for $i\in\llbracket 1, m\rrbracket$. Then \eqref{eq:abstract_fs_block} reduces to \eqref{eq:Block_Resolvent}, from which we conclude that the conditions \cref{lemma:FPE_resolvent_necessary} are necessary also for general fixed-point encoding as in \eqref{eq:abstract_fs_block}. Pick $\cH=\R$.

Item~\ref{lemma:FPE_splitting_necessary:HK}:
 Let $\bA = 0$ and let $C_i(x) = u^\star_i$ for $i\in \llbracket 1, m\rrbracket$ for any $\bu^\star\in\real^{m}$. The corresponding inclusion problem is solvable if and only if $\one^T\bu^\star = 0$ and there exists a fixed point if and only if there exists $\bz^\star \in \real^{n-1}$ and $\bx^\star\in \Ran(\one)$ satisfying
 \begin{equation}
        \left(\Gamma^{-1}-L+\bA\right) (\bx^\star) + H\bu^\star \in \Ran(M), \quad \text{i.e.,}\quad 0 + \one^TH\bu^\star = \one^TM = 0,
 \end{equation}
 due to the necessary conditions of Lemma~\ref{lemma:FPE_resolvent_necessary}.
 Item \ref{item:zero_empty_iff_fix_empty} in the fixed-point encoding property implies that $\one^TH\bu^\star = 0$ if and only if $\one^T\bu^\star = 0$, from which $\one^TH= \one^T$.
 
 Consider instead $C_i(x) = \beta_i\left(x - K_i\one\right)$ for a fixed $i\in \llbracket 1, m\rrbracket$, $C_j = 0$ for $j\in \llbracket 1, m \rrbracket \setminus \{i\}$ and let $A_j(x) = -(\Gamma^{-1}-L)_j\one x$ for $j\in \llbracket 1, n\rrbracket$ and all $x \in \R$. Observe that, using Lemma \ref{lemma:FPE_resolvent_necessary}\ref{lemma:FPE_resolvent_necessary:gamma}, $x^\star = K_i\one$ is the unique solution to the corresponding inclusion problem. Further, $(\bar \bz, \bar \bx,\bar \bu, \bar \ba)=(0, \one, 0, 0)$ solves \eqref{eq:abstract_fs_inclusion} at $\bar \bz$, since $\bar \bz^+ = U\bz-V\bx^+-Z^{(3)} \bu^+ = 0$,
 \begin{equation}
         \bar \bu = \bC(K\one + 0 + 0) = 0, \quad \text{and} \quad  0 = \left(\Gamma^{-1}-L + \bA\right)(\bar \bx) = M \bar\bz - H\bar \bu = 0.
\end{equation}
In particular, $\bar \bz \in \Fix(T)$. Therefore, Proposition \ref{prop:solution_map}\ref{item:prop_solution_map_characterization} yields $x^\star=1$, hence $K_i\one = 1$. As $i$ was arbitrary, we get $K\one = \one$.
     
Item~\ref{lemma:FPE_splitting_necessary:Z}: Let $A_i = \iota_{\{x^\star\}}$ for $i\in \llbracket 1,n\rrbracket$ and $x^\star\in \real$. Pick any $i\in \llbracket 1, m\rrbracket$, let $C_i = 1$  and let $C_j = 0$ for $j$ in $ \llbracket 1,m\rrbracket\setminus \{i\}$. The corresponding inclusion problem is solvable but there cannot exist a fixed point $\bz^\star \in \R^{n-1}$ with $\bx^\star =  \one x^\star$, unless
\begin{equation}
    \bz^\star = \bz^\star - V\bx^\star - Z^{(3)}_i, \quad \text{i.e.,} \quad Z^{(3)}_i = 0,
\end{equation}
where we used Lemma \ref{lemma:FPE_resolvent_necessary}\ref{lemma:FPE_resolvent_necessary:V}.
Since the index $i$ was arbitrary, we get $Z^{(3)} = 0$. Let instead $C_i = \beta_iI$ for some $i\in \llbracket 1, m\rrbracket$ and let $C_j = 0$ for $j\in \llbracket 1, m\rrbracket \setminus \{i\}$. For any $\bz^\star\in \real^{n-1}$ let $A_j(x):=(M\bz^\star - \beta_iH\diag(e_i)Z^{(1)}\bz^\star)_j$ for $j\in \llbracket 1,n\rrbracket$. We note that $(\bz^\star, \bx^\star, \bu^\star, \ba^\star):=(\bz^\star, 0, \beta_i\diag(e_i)Z^{(1)}\bz^\star, \bA(0))$ solves \eqref{eq:abstract_fs_inclusion} at $\bz^\star$ as
\begin{equation}
    \begin{aligned}    
  \left(\Gamma^{-1}-L + \bA\right)(0)&= \bA(0) = M\bz^\star -  \beta_iH\diag(e_i)Z^{(1)}\bz^\star= M\bz^\star - H\bu^\star, \\
  \bu^\star &= \bC(0 + Z^{(1)}\bz^\star + 0) = \beta_i\diag(e_i)Z^{(1)}\bz^\star.
    \end{aligned}
\end{equation}
Thus, $\bz^\star \in \Fix(T)$. However, $\bx^\star = 0$ does not solve the corresponding inclusion unless
\begin{equation}
    0 = \sum_{j=1}^nA_j(0) + \sum_{j=1}^mC_j(0) = \one^TM\bz^\star - \beta_i\one^TH\diag(e_i)Z^{(1)}\bz^\star = e_i^TZ^{(1)}\bz^\star.
\end{equation}
Since $\bz^\star$ and $i$ were arbitrary, we conclude that $Z^{(1)} = 0$. For the last counterexample, let $C_i(x) = 1$ and $C_j(x) = \beta_j x$ for all $x \in \R$ and some fixed $i<j$, $i,j\in \llbracket 1, n\rrbracket$, and let $C_h = 0$ for $h\in \llbracket 1, m\rrbracket \setminus \{i,j\}$. Let $\bu^\star := e_i + \beta_jZ^{(2)}_{ij}e_j$ and let $A_i(x) = -H_i\bu^\star$ for $i\in \llbracket 1,n\rrbracket$. We note that $(\bz^\star, \bx^\star, \bu^\star, \ba^\star):=(0, 0, e_i + \beta_jZ^{(2)}_{ij}e_j, -H\bu^\star)$ solves \eqref{eq:abstract_fs_inclusion} at $\bz^\star$ since
\begin{equation}
    \left(\Gamma^{-1}-L + \bA\right)(0) = -H\bu^\star = M\bz^\star - H\bu^\star, \quad \text{and} \quad \bu^\star=\mathbf{C}(0 + Z^{(2)}\bu^\star + 0).
\end{equation}
However, since $\bx^\star = 0$ does not solve the corresponding inclusion problem unless $0 = \sum_{i=1}^mC_i(0) + \sum_{i=1}^nA(0) = 1 - \one^TH\bu^\star = 1-1-\beta_jZ^{(2)}_{ij}$, fixed-point encoding is violated unless $Z_{ij}^{(2)} = 0$ whenever $i < j$. Note that $Z^{(2)}_{ij} = 0$ also for $i\geq j$ from Proposition~\ref{prop:causal_iff}. Thus, $Z^{(2)} = 0$.
\end{proof}

We next show that the necessary conditions are also sufficient, and the fixed-point operator with matrices satisfying the conditions of \cref{lemma:FPE_splitting_necessary} and Proposition~\ref{prop:causal_iff} thus gives an exact characterization of the methods satisfying \eqref{prop:splitting}, \eqref{prop:frugal}, \eqref{prop:minlift}, and $\eqref{prop:FPE}$.
\begin{theorem}
    \label{thm:FPE_iff}
    A frugal splitting operator with minimal lifting, as parameterized by \eqref{eq:abstract_fs_block} has the fixed-point encoding property if and only if it can be represented as
    \begin{equation}\label{eq:frugal_splitting_fpe}
    	\bz^+:= T(\bz), \quad \text{with} \quad \left\{\begin{aligned}
    		&\bx^+ = J_{\Gamma \bA}(\Gamma L \bx^+ - \Gamma H \bC(K \bx^+) + \Gamma M \bz),\\
    		&\bz^{+} = \bz -V\bx^+,
    	\end{aligned}\right.
    \end{equation}
	where $M^T, V\in \R^{(n-1)\times n}$, $\Gamma := \diag(\gamma) \succ 0$, $L\in \real^{n\times n}$, and $H, K^T \in \R^{n\times m}$ satisfy:
	\begin{enumerate}[label=(\roman*)]
		\item \label{thm:FPE_iff:VM} $\Null( M^T) = \Null( V) = \Ran(\one)$,
		\item \label{thm:FPE_iff:KH} $H$ and $K$ are causal in the sense of Definition~\ref{def:causal_pair} and satisfy $H^T\one = K\one = \one$,
		\item\label{thm:FPE_iff:Gamma} $L\in \real^{n\times n}$ is strictly lower triangular such that $\one^T\left(\Gamma^{-1}-L\right)\one = 0$.
	\end{enumerate}
        In this case, if $\bz \in \Fix(T)$, $\bx^\star=\one x^\star$ and $x^\star$ solves \eqref{eq:Inclusion}.
\end{theorem}
\begin{proof}
    Necessity follows directly from Proposition~\ref{prop:causal_iff} and \cref{lemma:FPE_splitting_necessary}. For sufficiency, assume that conditions \cref{thm:FPE_iff:VM}--\cref{thm:FPE_iff:Gamma} hold.
    Then by \cref{thm:FPE_iff:VM}, there exists a fixed point of $T$ if and only if there exists $x^\star\in \cH$ such that
    \begin{equation}
        \label{eq: FPE-intersection}
        \Ran(M)\cap \left(\Gamma^{-1}-L+\bA + H\bC K\right)(\one x^\star )\neq \emptyset.
    \end{equation}
    Since $\Ran(M) = (\one^T)^\perp$ and Items \ref{thm:FPE_iff:KH} and \ref{thm:FPE_iff:Gamma} hold, \eqref{eq: FPE-intersection} is in turn equivalent to
    \begin{equation}
        0\in \one^T\left(\Gamma^{-1}-L+\bA+H\bC K\right)(\one x^\star) = \sum_{i=1}^nA_i(x^\star) + \sum_{i=1}^mC_i(x^\star).
    \end{equation}
    There thus exists a fixed point of $T$ if and only if the inclusion problem is solvable, and at a fixed point $\bz^\star$ we obtain $\bx^\star := \one x^\star$ for $x^\star$ solving the inclusion problem.
\end{proof}

\section{Necessary and sufficient conditions for nonexpansiveness}\label{sec:conditions_for_nonexpansivness}

Theorem~\ref{thm:FPE_iff} exactly captures the class of all frugal splitting methods with minimal lifting that have the fixed-point encoding property. This property guarantees that fixed points of the algorithm correspond to solutions of the underlying inclusion problem. However, an iteration of the form $\bz^{k+1}=T(\bz^k)$, where $T$ is an operator as in \eqref{eq:frugal_splitting_fpe}, does not necessarily find a fixed point and is therefore not useful. In this section, we will give necessary and sufficient conditions on the algorithm matrices for that $T$ is averaged nonexpansive, implying convergence to a fixed point. Instead of imposing structural constraints on the involved matrices, like fixed-point encoding does, it constrains the sizes of involved matrices in the form of a linear matrix inequality. Note that we treat nonexpansiveness with respect to the canonical norm in $\cH^{n-1}$. As we have shown in \cref{sec:NE-norm_and_equivalence}, this entails no loss of generality.

\subsection{The case of pure resolvent splitting}\label{sec:nonexp_resolvent}

As in Section~\ref{sec:Conditions_for_fixed-point-encoding}, we first tackle the case of pure resolvent splitting, i.e., with $m=0$. So let us consider \eqref{eq:frugal_splitting_fpe} with $\bC= 0$:
\begin{equation}\label{eq:frugal_resolvent_fpe}
	 \bz^+:= T(\bz), \quad \text{with} \quad \left\{\begin{aligned}
		&\bx^+ = J_{\Gamma \bA}(\Gamma L \bx^+ + \Gamma M \bz),\\
		&\bz^{+} = \bz -V\bx^+.
	\end{aligned}\right.
\end{equation}
We start with the following necessary condition:
\begin{lemma}\label{lem:M_equals_V}
	If $T$ as defined in \eqref{eq:frugal_resolvent_fpe} is nonexpansive, then
	\begin{equation}
		V= \alpha M^T, \quad \text{for some} \  \alpha \in \R_{++}.
	\end{equation}
\end{lemma}
\begin{proof}
	Also in this case we argue in $\cH=\R$. We make the following choice: Fix an index $i \in \llbracket1, n \rrbracket$, and let $A_h = \partial f_h$ for all $h \in \llbracket 1, n\rrbracket$, with 
	\begin{equation}\label{eq:proof_lemma_V_M_1}
		f_i = 0, \quad \text{and} \quad f_j = \iota_{\{0\}}, \quad \text{for all} \ j \neq i.
	\end{equation}
	Since the proximity operator of $f_j$ with $j \neq i$ is simply the zero function, in this case, for each $\bz \in \R^{n-1}$, we get $x_i^+  =  \gamma_i m_i^T \bz$ and $x_j^+ = 0$ for all $j \neq i$,	where $m_i$ is the $i$th row of $M$. Therefore, denoting by $v_i$ the $i$th column of $V$, $T$ writes as
	\begin{equation}\label{eq:Z_equals_B_proof_1}
		T(\bz)  =  \bz - V \bx^+ =  \bz - \gamma_i v_i m_i^T \bz = \left(I - \gamma_i v_i m_i^T \right) \bz .
	\end{equation}
	Since $T$ is nonexpansive, $I- T = \gamma_i v_i m_i^T$ is monotone, implying $v_i = \alpha_i m_i$ for some $\alpha_i\geq 0$. Repeating the argument with $i \in \llbracket1, n\rrbracket$, we find that $V = M^T \diag(\alpha_1,\dots, \alpha_n)$. Since $\Null(V)=\Null(M^T)=\Ran(\one)$ from Lemma~\ref{lemma:FPE_resolvent_necessary}, it must hold $\alpha_1 = \dots = \alpha_n =\alpha$.
\end{proof}

\begin{remark}\label{remark:unconditional_stability}
	Assuming that $T$ is only unconditionally stable as done for the two-operators case in the seminal work \cite{Ryu}, raises important difficulties, as with the choice \eqref{eq:proof_lemma_V_M_1} we only obtain:
	\begin{equation*}
		0 \leq v_i m_i^T \leq 2 \gamma_i^{-1}, \quad \text{for all} \ i \in \llbracket1, n\rrbracket.
	\end{equation*}
	Whether $V = \theta M^T$ holds when $n \geq 3$ (the case $n = 2$ follows from Lemma~\ref{lemma:FPE_resolvent_necessary}) and only assuming unconditional stability, remains a challenging open problem. 
\end{remark}

\begin{remark}\label{remark:removing_alpha}
To drop the dependency on $\alpha$, from now on we can consider the equivalent algorithm $\bar T = \mathfrak{E} T (\mathfrak{E}^{-1})$ with $\mathfrak{E} := \alpha^{-\frac{1}{2}}I$, and change (without renaming) $\bar M := \alpha ^{\frac{1}{2}}M$ to $M$. Since $\mathfrak{E}$ is a scaling, the two operators are nonexpansive with respect to the same norm. Thus, we shall suppose without loss of generality that $V = M^T$.
\end{remark}

It remains to obtain necessary and sufficient conditions for \eqref{eq:frugal_resolvent_fpe} with $V=M^T$ to be averaged nonexpansive. Let us first highlight the following important remark:
\begin{remark}\label{remark:removing_theta}
	The class of averaged operators is in a one-to-one correspondence with that of nonexpansive operators. Indeed, by definition $T$ is a $\theta$-averaged operator if and only if $T_{(1)}:= \frac{1}{\theta}T - \frac{1-\theta}{\theta}\Id$ is nonexpansive. Therefore, in our setting, to characterize the class of $\theta$-averaged operators, it suffices to characterize the class of nonexpansive operators and scale them as $\theta T + (1-\theta)\Id$ to obtain the full class of $\theta$-averaged operators.
\end{remark}

\begin{theorem}
	\label{thm:frugal_resolvent_iff}
	A frugal resolvent operator with minimal lifting, as parameterized by \eqref{eq:abstract_fs_block} has the fixed-point encoding property and is $\theta$-averaged if and only if it can be written as
	\begin{equation}\label{eq:frugal_splitting_nonexp}
		 \bz^+:= T(\bz), \quad \text{with} \quad \left\{\begin{aligned}
			&\bx^+ = J_{\Gamma \bA}(\Gamma L \bx^+ + \Gamma M \bz),\\
			&\bz^{+} = \bz - \theta M^T\bx^+,
		\end{aligned}\right.
	\end{equation}
	where $M^T \in \R^{(n-1)\times n}$, $\Gamma := \diag(\gamma)\succ 0$, and $L\in \real^{n\times n}$ satisfy:
	\begin{enumerate}[label=(\roman*)]
		\item\label{thm:resolvent_nonexp_iff_1} $\Null (M^T) = \Ran(\one)$,
		\item\label{thm:resolvent_nonexp_iff_2} $L\in \real^{n\times n}$ is strictly lower triangular such that $\one^T\left(\Gamma^{-1}-L\right)\one = 0$,
		\item\label{thm:resolvent_nonexp_iff_3} $2\Gamma^{-1}-L-L^T\succeq  MM^T$.
	\end{enumerate}
\end{theorem}
\begin{proof}
	First, from Lemma~\ref{lem:M_equals_V} and Remark~\ref{remark:removing_alpha} we can assume that $V=M^T$. Second, from Remark~\ref{remark:removing_theta}, we can suppose that $\theta = 1$ and focus on the nonexpansivity of $T$. Recall that for all $\bz \in \cH^{n-1}$ evaluating $T(\bz)$ yields $\bx^+$ and $\ba^+ \in \bA(\bx^+)$ such that $(\Gamma^{-1} - L)\bx^+ + \ba^+ = M\bz$. Therefore, for all $\bz, \bar \bz \in \cH^{n-1}$, denoting by $\bx^+$, $\ba^+$, and $\bar \bx^+$, $\bar \ba^+$ the corresponding elements generated by evaluation of $T(\bz)$ and $T(\bar \bz)$, respectively, we get:
    \begin{equation}\label{eq:frugal_resolvent_iff_nonexpasive}
		\begin{aligned}
			&\|T(\bz) - T(\bar \bz)\|^2 - \|\bz - \bar \bz\|^2 =  \|\bz - \bar \bz - M^T(\bx^+  - \bar \bx^+)\|^2 - \|\bz - \bar \bz\|^2 = \\
			& = -2\langle M\bz- M\bar \bz, \bx^+  - \bar \bx^+\rangle + \|\bx^+  - \bar \bx^+\|^2_{MM^T}\\
			& = -2 \langle (\Gamma^{-1} - L)(\bx^+ - \bar \bx^+), \bx^+ - \bar \bx^+\rangle - \langle \ba - \bar \ba, \bx^+ - \bar \bx^+\rangle + \|\bx^+  - \bar \bx^+\|^2_{MM^T},\\
			& = - \langle (2\Gamma^{-1} - L - L^T)(\bx^+ - \bar \bx^+), \bx^+ - \bar \bx^+\rangle - \langle \ba - \bar \ba, \bx^+ - \bar \bx^+\rangle + \|\bx^+  - \bar \bx^+\|^2_{MM^T}.
		\end{aligned}
	\end{equation}
	
	\textit{Sufficiency:} Using that $\bA$ is monotone, the right hand-side of \eqref{eq:frugal_resolvent_iff_nonexpasive} is non-positive if the quadratic form in $\bx^+ - \bar \bx^+$ is non-positive, which is indeed a consequence of Item~\ref{thm:resolvent_nonexp_iff_3}.
	
	\textit{Necessity:} Items \ref{thm:resolvent_nonexp_iff_1} and \ref{thm:resolvent_nonexp_iff_1} follow from Items \ref{thm:FPE_iff:Gamma} and \ref{thm:FPE_iff:VM} of Theorem \ref{thm:frugal_resolvent_iff}. Let us show Item~\ref{thm:resolvent_nonexp_iff_3}. If $T$ is nonexpansive for all $A_1, \dots, A_n$ with $\zer(\sum_{i=1}^n A_i) \neq \emptyset$ the right hand-side of \eqref{eq:frugal_resolvent_iff_nonexpasive} should be nonpositive for each of these possible choices. Let $\cH=\R$. We proceed with a Sylvester-type argument\footnote{Unfortunately, picking $\bA=0$ does not yield the claim since $\bx^+ = (\Gamma^{-1} - L)^{-1}M\bz$ and thus $\bx^+ - \bar \bx^+$ would only range on a subspace of $\R^{n}$, as $\bz$ and $\bar \bz$ vary in $\R^{n-1}$.}: Pick $\mathcal{I} \subset \llbracket1, n\rrbracket$ with $|\mathcal{I}| \leq n-1$ and let $A_i := \partial f_i$ with
	\begin{equation}
		f_i := 0, \quad \text{for} \ i \in \mathcal{I}, \quad \text{and} \quad f_i :=\iota_{\{0\}} , \quad \text{for} \ i \notin \mathcal{I}.
	\end{equation}
 	Let us denote $\mathcal{J}:=\mathcal{I}^c$. In this setting, evaluating $T$ yields $x_i^{+} = 0$ for all $i \notin \mathcal{I}$ and $a_i^+ = 0$ for all $i \in \mathcal{I}$. Using that $A_i =\partial \iota_{\{0\}} = \R$ for all $i \notin \mathcal{I}$, the other coordinates, are those (and only those) that solve the following system:
	\begin{equation}\label{eq:frugal_resolvent_iff_evaluation}
		\bx_{\mathcal{I}}^+ \in \R^{|\mathcal{I}|}, \quad \ba_{\mathcal{J}}^+\in \R^{|\mathcal{J}|} \quad \text{such that}: \quad \left\{\begin{aligned}
		&(\Gamma^{-1} - L)_{\mathcal{I} \mathcal{I}}\bx_{\mathcal{I}}^+ = M_{\mathcal{I}} \bz,\\
		& (\Gamma^{-1} - L)_{\mathcal{J}  \mathcal{I}}\bx_{\mathcal{I}}^+ + \ba_{\mathcal{J}}^+= M_{\mathcal{J}} \bz,
		\end{aligned}\right.
	\end{equation}
	where $(\Gamma^{-1} - L)_{\mathcal{J} \mathcal{I}}$ denotes the minor of $\Gamma^{-1} - L$ obtained keeping only its $\mathcal{J}$ rows and its $\mathcal{I}$ columns, and $M_{\mathcal{I}}$ denotes the matrix obtained by keeping its $\mathcal{I}$ rows. Now, since $\Gamma^{-1} - L$ is invertible and lower-triangular, also $(\Gamma^{-1} - L)_{\mathcal{I} \mathcal{I}}$ is so, and since $\mathcal{M}_{I}$ is full rank, then $(\Gamma^{-1} - L)_{\mathcal{I} \mathcal{I}}^{-1}M_{\mathcal{I}}$ is well-defined and onto. It follows that the unique solution to \eqref{eq:frugal_resolvent_iff_evaluation} is
	\begin{equation}\label{eq:frugal_resolvent_iff_nonexpasive_sols}
		\bx_{\mathcal{I}}^+ = (\Gamma^{-1} - L)_{\mathcal{I} \mathcal{I}}^{-1}M_{\mathcal{I}}\bz, \quad \text{and} \quad \ba_{\mathcal{J}}^+ = M_{\mathcal{J}} \bz - (\Gamma^{-1} - L)_{\mathcal{J}  \mathcal{I}}\bx_{\mathcal{I}}^+.
	\end{equation}
 	By the zero patterns of $\bx^+$ and $\ba^+$, the nonexpansiveness of $T$ in \eqref{eq:frugal_resolvent_iff_nonexpasive} now reads as:
	\begin{equation}
		- \langle (2\Gamma^{-1} - L - L^T)_{\mathcal{I}\mathcal{I}}(\bx_\mathcal{I}^+ - \bar \bx_\mathcal{I}^+), \bx_\mathcal{I}^+  - \bar \bx_\mathcal{I}^+\rangle + \|\bx_\mathcal{I}^+  - \bar \bx_\mathcal{I}^+\|^2_{(MM^T)_{\mathcal{I}\mathcal{I}}} \leq 0, \quad \text{for all} \ \bz, \bar \bz \in \R^{n-1}.
	\end{equation} 
	Using now \eqref{eq:frugal_resolvent_iff_nonexpasive_sols} and the fact that $(\Gamma^{-1} - L)_{\mathcal{I} \mathcal{I}}^{-1}M_{\mathcal{I}}$ is onto, we deduce that 
	\begin{equation}
	(2\Gamma^{-1} - L - L^T - MM^T)_{\mathcal{I}\mathcal{I}} \succcurlyeq 0, \quad \text{for all} \ \mathcal{I}\subset \llbracket 1, n\rrbracket, \ |\mathcal{I}|\leq n-1. 
	\end{equation}
	Gathering $\one^T\left(\Gamma^{-1}-L\right)\one = 0$ (Item~\ref{thm:resolvent_nonexp_iff_2}), and $\one^TMM^T\mathbf{1}=0$ (Item~\ref{thm:resolvent_nonexp_iff_1}), we get
	\begin{equation*}
		\one^T(2\Gamma^{-1}-L-L^T - MM^T)\one=0, \quad \text{hence} \quad \det(2\Gamma^{-1}-L-L^T - MM^T)=0.
	\end{equation*}
	Since each minor of $2\Gamma^{-1}-L-L^T - MM^T$ has non-negative determinant,  Sylvester's criterion for positive semidefinite matrices yields the claim.
\end{proof}

\begin{remark}
 Inspecting the proof of Theorem~\ref{thm:frugal_resolvent_iff}, we can notice that the class of test operators (or counterexamples) reduces to (sub)gradients of indicator, linear and quadratic functions, even with $\cH=\R$. We shall conclude that \eqref{eq:frugal_splitting_nonexp} defines the full class of averaged resolvent splitting with minimal lifting even to address \eqref{eq:Inclusion} with respect to this particularly coarse family of operators in dimension 1.
\end{remark}

\subsection{The general case}

In this section, we now focus on the general case with $m\geq 1$. The parameterization of \eqref{eq:frugal_splitting_fpe} characterizes all methods satisfying properties \eqref{prop:splitting}--\eqref{prop:FPE}, as a consequence of Proposition~\ref{prop:causal_iff} and Theorem~\ref{thm:FPE_iff}. Further, from Lemma~\ref{lem:M_equals_V} and Remark~\ref{remark:removing_alpha} we can suppose that $T$ is of the form:
\begin{equation}\label{eq:frugal_splitting_fpe_plus}
	\bz^+:= T(\bz), \quad \text{with} \quad \left\{\begin{aligned}
		&\bx^+ = J_{\Gamma \bA}\left(\Gamma L \bx^+ - \Gamma H \bC(K \bx^+) + \Gamma M \bz\right),\\
		&\bz^{+} = \bz - M^T\bx^+.
	\end{aligned}\right.
\end{equation}
To tackle this more general setting, it turns out to be convenient to divide our analysis into two parts. First, we give sufficient conditions, and then, we show that these are indeed necessary as well under one additional dimensionality requirement. 

\subsubsection{Sufficient conditions for nonexpansiveness}\label{sec:sufficient_conditions}

We introduce the following assumption which we will show to be necessary and sufficient for averaged nonexpansiveness as long we are not constrained to inclusion problems in $\cH$ with $\dim(\cH)<2n+m-1$:
\begin{assumption}
    Suppose that the following holds for $\beta:=(\beta_1, \dots, \beta_m)$
    \begin{equation}\label{ass:NE}
    	2\Gamma^{-1}-L-L^T\succeq  MM^T + \frac{1}{2}\left(H-K^T\right)\diag(\beta)\left(H^T-K\right).
    \end{equation}
\end{assumption}

Sufficiency of \eqref{ass:NE} follows from a direct analysis, without any further assumptions on $\dim(\cH)$, much like in the sufficiency part of Theorem~\ref{thm:frugal_resolvent_iff}.
\begin{lemma}
    \label{lemma:NE_sufficient}
    Let $T$ be a frugal splitting parameterized by \eqref{eq:frugal_splitting_fpe_plus}. If Assumption~ \eqref{ass:NE} holds, then $T$ is nonexpansive for every $A\in \cA_n$, $C \in \cC_m$, and all real Hilbert spaces $\cH$.
\end{lemma}
\begin{proof}
    Let $\bz \in \mathcal{H}^{n-1}$. Recall that evaluating $T(\bz)$ produces $\ba^+ \in \bA(\bx^+)$, $\bx^+ \in \cH^{n}$, and $\bu^+ =\bC(K\bx^+)$ such that:
    \begin{equation}
    	(\Gamma^{-1} - L) \bx^+ + \ba^+ + H \bu^+ = M \bz.
    \end{equation}
    Pick two arbitrary $\bz, \bar \bz \in \cH^{n-1}$, and let $\bx^+$, $\ba^+$, $\bu^+$ and $\bar \bx^+$, $\bar \ba^+$, $\bar \bu^+$ be the corresponding terms obtained by evaluation of $T(\bz)$ and $T(\bar \bz)$. Observe that:
    \begin{equation}
        \begin{aligned}
            &\|T(\bz)-T(\bar\bz)\|^2 - \|\bz-\bar\bz \|^2 = \|(\bz-\bar\bz)- M^T(\bx^+-\bar\bx^+) \|^2-\|\bz-\bar\bz \|^2 \\
            &\leq
            \|\bx^+-\bar \bx^+\|_{MM^T}^2 - 2\langle M(\bz-\bar\bz),  \bx^+-\bar\bx^+ \rangle\\
            & \quad + 2\left(\langle \bu^+-\bar \bu^+, K(\bx^+-\bar \bx^+)\rangle - \|\bu^+-\bar\bu^+ \|^2_{\diag(\beta)^{-1}}\right)\\
            &= -\|\bx^+-\bar\bx^+ \|_{2\Gamma^{-1}-L-L^T - MM^T}^2 -\langle \ba^+ - \bar \ba^+, \bx^+ - \bar \bx^+\rangle - 2\langle \bx^+-\bar\bx^+, H(\bu^+-\bar\bu^+)\rangle \\
            &\quad + 2 \langle \bu^+-\bar \bu^+, K(\bx^+-\bar \bx^+\rangle -2\|\bu^+-\bar\bu^+\|^2_{\diag(\beta)^{-1}} \leq - \|(\bx-\bar\bx, \bu-\bar\bu) \|_Q^2 \leq 0,
        \end{aligned}
    \end{equation}
    where the first equality uses the definition of $T$, the 
    first inequality uses the nonnegativity of:
        \begin{equation*}
        \langle \bu^+-\bar\bu^+, K(\bx^+-\bar\bx^+)\rangle - \|\bu^+-\bar\bu^+\|_{\diag(\beta)^{-1}}^2  =\sum_{i=1}^m\left(\langle \bu_i^+-\bar \bu_i^+, K_i(\bx^+-\bar \bx^+)\rangle - \frac{1}{\beta_i}\|\bu_i^+-\bar \bu_i^+ \|^2\right),
    \end{equation*}
    which follows from $\frac{1}{\beta}_i$-cocoercivity of $C_i$ for $i\in \llbracket 1, m\rrbracket$, and the second inequality uses monotonicity of $\bA$. The last inequality follows from the positive semidefiniteness of the matrix
    \begin{equation}\label{eq:def_Q}
    	Q := \begin{bmatrix}
    		2\Gamma^{-1}-L - L^T- MM^T & H-K^T\\
    		H^T-K & 2\diag(\beta)^{-1}
    	\end{bmatrix}.
    \end{equation}
    which is indeed equivalent, by Schur complement, to \eqref{ass:NE}.
\end{proof}

\subsubsection{Necessary and sufficient conditions}
We now show necessity of the condition in \eqref{ass:NE} for nonexpansiveness. In the proof of Theorem~\ref{thm:frugal_resolvent_iff}, we design several maximal monotone operators $A_1, \dots, A_n$ to test the nonexpansiveness inequality \eqref{eq:frugal_resolvent_iff_nonexpasive} and obtain each time specific properties that eventually yield Item~\ref{thm:resolvent_nonexp_iff_3}. Unfortunately, in the case $\bC \neq 0$, this approach is no longer efficient as simple choices of $C_1, \dots, C_m$, e.g., linear operators, seem to yield weaker conditions than \eqref{ass:NE}. To overcome this apparent lack of simple counterexamples, we consider the following optimization problem:
\begin{equation}\label{eq: PEP-1}
        \begin{aligned}
                &\underset{A, C, \bz, \bz^\star}{\text{maximize}} && \|T(\bz)-T(\bz^\star)\|^2-\|\bz-\bz^\star \|^2, \\
                &\text{subject to} && A \in \cA_n, \ C\in \cC_m, \  \text{$\textstyle \zer \big(\sum_{i=1}^n A_i + \sum_{i=1}^m C_i\big)\neq \emptyset$}, \ \bz, \bz^\star\in \mathcal{H}^{n-1}.
            \end{aligned}
\end{equation}
This problem will be reformulated using interpolation conditions, where the tightness of the cocoercive interpolation conditions are shown in \cite{RyuOperatorSplittingPEP}, and the resulting SDP-relaxation will be tight assuming $\dim(\cH)\geq 2n+m-1$, a condition not appearing in the pure resolvent splitting in Section~\ref{sec:nonexp_resolvent}. We will lastly find that feasibility of the dual SDP is equivalent both to nonexpansiveness, by strong duality, and to assumption \eqref{ass:NE}. 

We begin with a nonconvex, finite-dimensional reformulation of Problem~\eqref{eq: PEP-1}, whose proof is deferred to the appendix. To ensure \( \Fix(T) \neq \emptyset \) by construction, we focus on the nonexpansiveness of \( T \) relative to its fixed points---i.e., \emph{quasinonexpansiveness}. Interestingly, this requirement imposes no additional restrictions, and, in fact, even the techniques employed to prove Theorem \ref{thm:frugal_resolvent_iff} trivially extend to this more general setting.
    \begin{lemma}
        \label{lemma:NE_Primal_necessary}
        Consider the frugal splitting operator $T$ defined by \eqref{eq:frugal_splitting_fpe_plus}. $T$ is quasinonexpansive for all $A \in \cA_n$, $C \in \cC_m$ and all real Hilbert space $\cH$ if and only if the following optimization problem is non-positive
            \begin{equation}
            \label{eq: PEP-2}
            \begin{aligned}
                &\underset{\bx, \bx^\star, \bu,  \bu^\star, \bz, \bz^\star}{\text{maximize}} & & \|\bz-\bz^\star - M^T(\bx-\bx^\star)\|^2-\|\bz-\bz^\star \|^2, \\
                &\text{subject to} & & \langle x_i-x_i^\star, -(\Gamma^{-1} -L)_i(\bx-{\bx}^\star) - H_i(\bu-{\bu}^\star) + M_i(\bz-{\bz}^\star)\rangle \geq 0, \quad \forall i\in \llbracket 1, n\rrbracket,\\
                &&& \langle u_i-u_i^\star, K_i(\bx-{\bx}^\star)\rangle - \frac{1}{\beta_i} \|u_i-u^\star_i\|^2 \geq 0, \quad \forall i\in \llbracket 1, m\rrbracket, \\
                &&& \bz^\star = T(\bz^\star), \ \bx, \bx^\star \in \mathcal{H}^{n}, \ \bu, \bu^\star\in \mathcal{H}^m, \ \bz, \bz^\star\in \mathcal{H}^{n-1}.
            \end{aligned}  
        \end{equation}
        \end{lemma}
        
    This problem depends only on $\bz-\bz^\star$, $\bx-\bx^\star$, $\bu-\bu^\star$, so without loss of generality we can set $(\bz^\star, \bx^\star, \bu^\star) = (0,0,0)$, in which case $\bz^\star = T(\bz^\star)$ holds trivially.
    The following convex SDP-formulation is equivalent to \eqref{eq: PEP-2} if $\dim(\cH)\geq 2n+m-1$, cf.~Appendix \ref{app:missing_proofs} for the proof.
    \begin{lemma}\label{lemma:NE_Dual}
            Under the assumption that $\dim(\mathcal{H})\geq 2n+m-1$, non-positivity of \eqref{eq: PEP-2} is equivalent to non-positivity of the following SDP.
    \begin{equation}
        \label{eq: SDP-PEP}
        \begin{aligned}
            &\underset{G \succeq 0}{\text{maximize}} &&  \Tr(AG),\\
            &\text{subject to} && \Tr(GE_i)\geq 0, \quad \text{for} \ i\in \llbracket1, n\rrbracket,\\
            &&&\Tr(GF_i)\geq 0, \quad \text{for} \ i\in \llbracket 1, m\rrbracket,
        \end{aligned}
    \end{equation}
    where, denoting by $e_1, \dots, e_n$ the Euclidean basis in $\R^n$:
    \begin{equation*}     
        A:= \begin{bmatrix}
            0 & -M^T & 0\\
            -M & MM^T & 0\\
            0 & 0 & 0
        \end{bmatrix}, \quad E^C_i:= \begin{bmatrix}
        0 & 0 & 0\\
        0 & 0 & K^T\diag(e_i)\\
        0 & \diag(e_i)K & -2\diag(e_i)\diag(\beta)^{-1}
    \end{bmatrix},
	\end{equation*}
	\begin{equation*}
	E^A_i := \begin{bmatrix}
		0 & M^T\diag(e_i) & 0\\
		\diag(e_i)M & -\diag(e_i)(\Gamma^{-1}-L) - (\Gamma^{-1}-L^T)\diag(e_i) & -\diag(e_i)H\\
		0 & -H^T\diag(e_i) & 0
	\end{bmatrix}.
	\end{equation*}
	\end{lemma}
    Let us now consider the dual problem of \eqref{eq: SDP-PEP}. Denote by
    \begin{equation}\label{eq:definition_A_lambda}
    	A_{\boldsymbol{\lambda}}:= A + \sum_{i=1}^{n} \lambda^A_iE_i^A + \sum_{i=1}^m\lambda^C_iE_i, \quad \text{for} \ \boldsymbol{\lambda}:=(\lambda^A, \lambda^C) \in \R^{n}_+ \times \R^m_{+}.
    \end{equation}
     Then, the dual problem of \eqref{eq: SDP-PEP} can be written as:
    \begin{equation}
        \begin{aligned}
            \label{eq: PEP-Dual}
            \inf_{\boldsymbol{\lambda} \in \R^{n}_+ \times \R^m_{+}} \sup_{ G\succeq 0} \ \Tr \left(G A_{\boldsymbol{\lambda}} \right) = \inf_{\boldsymbol{\lambda} \in \R^{n}_+ \times \R^m_{+}} \bigg\{ 0 \ : \  A_{\boldsymbol{\lambda}} \preceq 0\bigg\}.
        \end{aligned}
    \end{equation}
\begin{lemma}\label{lem:feasibility_dual}
The dual SDP \eqref{eq: PEP-Dual} is feasible if and only if \eqref{ass:NE} holds.
\end{lemma}
\begin{proof}
For $\boldsymbol{\lambda}=(\lambda^A, \lambda^C)\in\R^{n}_+ \times \R^m_{+}$ let $\Lambda^A := \diag(\lambda^A)$ and $\Lambda^C := \diag(\lambda^C)$. Then the dual SDP is feasible if and only if
	\begin{equation}
		\label{eq: Dual-matrix}
			0 \preceq -A_{\boldsymbol{\lambda}}  =\begin{bmatrix}
				0 & M^T(I - \Lambda^A) & 0\\
				(I - \Lambda^A)M &\Lambda^A(\Gamma^{-1}-L) + (\Gamma^{-1}-L^T)\Lambda^A -MM^T & \Lambda^AH-K^T\Lambda^C\\
				0 & H^T\Lambda^A - \Lambda^CK& 2\Lambda^C\diag(\beta)^{-1}
			\end{bmatrix}
	\end{equation} 
	for some $\boldsymbol{\lambda}$. First, we must have $\Lambda^A=I$ as the upper left block of \eqref{eq: Dual-matrix} is zero. Furthermore let $v:= \one-\lambda^C$ and consider test vectors of the form $x := \begin{bmatrix}
		0 &\one_n^T & -cv^T
	\end{bmatrix}^T$ for $c\in \real_+$. Then
	\begin{equation}
		\label{eq:Dual_C}
		\begin{aligned}
			x^TA_{\boldsymbol{\lambda}} x &= 
			0 + 2c^2v^T \diag(\beta)^{-1} v - 2c\one^T\left(\Lambda^AH-K^T\Lambda^C\right)v \\
			&= 2c^2v^T \diag(\beta)^{-1} v - 2c\left( \one^T-(\lambda^C)^T\right)v = 2c^2\|v\|^2_{\diag(\beta)^{-1}} - 2c\|v\|^2,
		\end{aligned}
	\end{equation} 
	where the first term of the first equality is zero, since $\one^T(\Gamma^{-1}-L)\one = 0$, $M^T\one = 0$ and $\lambda^A = \one$. And the second equality uses the $\one^T\Lambda^A = \lambda^{A^T}$, $\one^T\Lambda^C = \lambda^{C^T}$ and $\one^TH=\one^TK^T = \one^T$.
	The quadratic expression \eqref{eq:Dual_C} is negative for $c = \|v\|^2 (2\|v\|^2_{\diag(\beta)^{-1}})^{-1}$ unless $v = 0$. We conclude that the condition $\lambda^C = \one$ is necessary for positive semidefiniteness of $A_{\boldsymbol{\lambda}}$. With these restrictions, $A_{\boldsymbol{\lambda}}$ is is positive semidefinite if and only if the matrix $Q$ defined as in \eqref{eq:def_Q} is positive semidefinite, which is in turn equivalent to \eqref{ass:NE} by Schur complement.
\end{proof}

We now show that without additional assumptions, the optimal value of \eqref{eq: PEP-1} equals the optimal value of \eqref{eq: PEP-Dual}, i.e., the dual problem is feasible. Once again, the technical proof is postponed to the appendix.
\begin{lemma}\label{lemma:Slater}
    Let $H, K, \Gamma$ and $L$ matrices parametrizing \eqref{eq:frugal_splitting_fpe_plus}. Then, strong duality holds of the SDP in \eqref{eq: SDP-PEP} by Slater's condition.
\end{lemma}

In summary, we have shown that quasinonexpansiveness of $T$ for every monotone inclusion problem which can be formulated as in \eqref{eq: PEP-1}, is equivalent to non-positivity of \eqref{eq: SDP-PEP}, under the dimension assumptions. By Slater's condition we have shown strong duality, from which we conclude that the fixed-point iteration is quasinonexpansive for every inclusion problem if and only if Assumption~\ref{ass:NE} is satisfied. We are ready to establish the main result of this paper:

\begin{theorem}\label{thm:frugal_splitting_iff}
	If $\dim(\cH)\geq 2n+m-1$, a frugal splitting operator with minimal lifting $T$ as parameterized by \eqref{eq:abstract_fs_block} has the fixed-point encoding property and is $\theta$-averaged if and only if it can be written as
	\begin{equation}\label{eq:algorithm_1_nonpar}
		\bz^+:= T(\bz), \quad \text{with} \quad \left\{\begin{aligned}
			&\bx^+ = J_{\Gamma \bA}(\Gamma L \bx^+ - \Gamma \bC(K \bx^+)+ \Gamma M \bz),\\
			&\bz^{+} = \bz - \theta M^T\bx^+,
		\end{aligned}\right.
	\end{equation}
	where $M^T \in \R^{(n-1)\times n}$, $\Gamma := \diag(\gamma)\succ 0$, $L\in \real^{n\times n}$ and $H, K^T \in \R^{n\times m}$ satisfy:
	\begin{enumerate}[label=(\roman*)]
		\item\label{thm:splitting_nonexp_iff_1} $\Null (M^T) = \Ran(\one)$,
		\item\label{thm:splitting_nonexp_iff_2} $L\in \real^{n\times n}$ is strictly lower triangular such that $\one^T\left(\Gamma^{-1}-L\right)\one = 0$,
		\item\label{thm:splitting_nonexp_iff_3} $H$ and $K$ are causal in the sense of Definition~\ref{def:causal_pair} and satisfy $H^T\one = K\one = \one$,
		\item\label{thm:splitting_nonexp_iff_4} $2\Gamma^{-1}-L-L^T\succeq MM^T + \frac{1}{2}\left(H-K^T\right)\diag(\beta)\left(H^T-K\right)$.
	\end{enumerate}
\end{theorem}
\begin{proof}
    From Theorem \ref{thm:FPE_iff}, the first three conditions are necessary and sufficient for $T$ being a fixed-point encoding. Thus, sufficiency follows from Lemma \ref{lemma:NE_sufficient}. Let $C=0$, then $V=M^T$ is necessary from Lemma \ref{lem:M_equals_V} and Remark \ref{remark:removing_alpha}. It only remains to establish Item \ref{thm:splitting_nonexp_iff_4} as a necessary conditions. This follows combining Lemma \ref{lemma:NE_Primal_necessary}, \ref{lemma:NE_Dual}, \ref{lem:feasibility_dual} and \ref{lemma:Slater}.
\end{proof}

Since $T$ is a $\theta$-averaged operator with $\theta \in (0, 1)$, whenever \eqref{eq:Inclusion} admits a solution, the corresponding fixed-point sequence converges weakly to some $\bz^\star \in \Fix(T)$. From Proposition \ref{prop:solution_map}, we know that the corresponding $\bx^\star$ obtained through evaluation of $T(\bz^\star)$ is such that $\bx^\star =\one x^\star$ with $x^\star$ solution to \eqref{eq:Inclusion}. However, the weak convergence of the sequence \( \{\bz^k\}_k \) does not, in general, imply the weak convergence of each sequence \( \{x_i^{k+1}\}_k \) for \( i \in \llbracket 1, n \rrbracket \), since resolvents are not necessarily weakly continuous. The following result bridges this gap by establishing that all sequences \( \{x_i^{k+1}\}_k \) converge weakly to the same solution \( x^\star \) of~\eqref{eq:Inclusion}, without requiring additional assumptions. Moreover, we derive a rate of convergence to consensus, quantified through the \emph{variance}, defined as:
\begin{equation}
	\Var (\bx) := \frac{1}{n} \sum_{i=1}^n \|x_i - \bar{x}\|^2, \quad \text{where} \quad \bar{x} := \frac{1}{n} \sum_{i=1}^n x_i, \quad \text{for any} \ \bx := (x_1, \dots, x_n) \in \cH^n.
\end{equation}

\begin{proposition}\label{prop:convergence_algorithm_1}
Assume that there exists a solution to \eqref{eq:Inclusion}. Let $\{\bz^{k}\}_k$ be the sequence generated by a frugal splitting method parameterized as in \eqref{eq:algorithm_1_nonpar} and $\{\bx^{k+1}\}_k$ the corresponding sequence obtained by evaluation of $T$. Then,
	\begin{enumerate}[label=(\roman*)]
		\item\label{item:convergence_algorithm_1_1} For all $i \in \llbracket1, n\rrbracket$, $x_i^{k+1}\rightharpoonup x^\star$, where $x^\star$ solves \eqref{eq:Inclusion},
		\item\label{item:convergence_algorithm_1_2} $\bz^k \rightharpoonup \bz^\star$ such that $\bx^\star:=x^\star \one = J_{\Gamma \bA}(\Gamma L \bx^\star - \Gamma H\bC(K \bx^\star)+ \Gamma M \bz^\star)$,
		\item\label{item:convergence_algorithm_1_3} $\Var(\bx^{k+1}) = o(k^{-1})$ as $k\to +\infty$.
	\end{enumerate}
\end{proposition}
The proof of Proposition~\ref{prop:convergence_algorithm_1} follows the same \emph{degenerate preconditioned proximal point} approach employed in \cite{bcn24}, and we provide it in the appendix. Note that employing a Fast-Krasnoselskii--Mann approach \cite{bn23}, as done in \cite{bcf24} for the pure resolvent case, we can also obtain the improved rate $\Var(\bx^{k+1}) = o(k^{-2})$ as $k\to +\infty$ at the cost of sacrificing variable minimality. 

In the following section, we show how to turn \eqref{eq:algorithm_1_nonpar} into Algorithm~\ref{alg:simple_splitting_FPE_introduction}, a general family of algorithm, indeed as general as \eqref{eq:algorithm_1_nonpar}, but for which Item~\ref{thm:splitting_nonexp_iff_4} is satisfied by construction. We refer to it as a \emph{parameterization} of the family of algorithms in \eqref{eq:algorithm_1_nonpar} as it only consists in choosing $3$ sets of matrices.

\section{Algorithm 1: Description, reformulation and special cases}\label{sec:Problem_formulation_and_parameterization}

We first show that Algorithm~\ref{alg:simple_splitting_FPE_introduction} and \eqref{eq:algorithm_1_nonpar} cover the same family of algorithms:
\begin{proposition}\label{prop:equivalence_with_alg_1}
	A frugal splitting operator with minimal lifting as in \eqref{eq:abstract_fs_block} has the fixed-point encoding property and is $\theta$-averaged if and only if it can be defined as in Algorithm~\ref{alg:simple_splitting_FPE_introduction}.	
\end{proposition}
\begin{proof}
	To show the result, we only need to show that from any instance of Algorithm~\ref{alg:simple_splitting_FPE_introduction} we obtain an instance of \eqref{eq:algorithm_1_nonpar} and vice-versa. So let us consider an instance of \eqref{eq:algorithm_1_nonpar} parametrized by $M$, $\Gamma$, $L$ and $H, K$.  We need to define $S$ and $P$. To do so, we pick
	\begin{equation*}
		S:= 2 \Gamma^{-1}- L - L^T, \quad \text{and} \quad \bar P:= S - MM^T - \frac{1}{2}(H - K^T)\diag(\beta) (H^T - K),
	\end{equation*}
	and take any $P\in \R^{n\times (n-1)}$ with $\Ran(\one) \subset \Null(P^T)$ such that $PP^T = \bar P$, which is possible since from Theorem~\ref{thm:frugal_splitting_iff}\ref{thm:splitting_nonexp_iff_1}, \ref{thm:splitting_nonexp_iff_2} and \ref{thm:splitting_nonexp_iff_3} $\Ran(\one) \in \Null(\bar P)$ and $\bar P \succcurlyeq 0$ from Theorem~\ref{thm:frugal_splitting_iff}\ref{thm:splitting_nonexp_iff_4}. The matrices $M$, $P$, $H, K$ and $S$ provide an instance of Algorithm~\ref{alg:simple_splitting_FPE_introduction}.
	
	Conversely, if $M$, $P$, $H, K$ and $S$ parametrize Algorithm~\ref{alg:simple_splitting_FPE_introduction}, define
	\begin{equation*}
		L := -\slt(S), \quad \text{and} \quad \Gamma := 2\diag(S)^{-1}. 
	\end{equation*}
	Then, Algorithm~\ref{alg:simple_splitting_FPE_introduction} writes as in \eqref{eq:algorithm_1_nonpar}, and since $2\Gamma^{-1} - L - L^T = S$ we have Theorem~\ref{thm:frugal_splitting_iff}\ref{thm:splitting_nonexp_iff_1}, \ref{thm:splitting_nonexp_iff_2}, while Item~\ref{thm:splitting_nonexp_iff_3} follows by construction since $S\succeq MM^T + \frac{1}{2}(H - K^T)\diag(\beta) (H^T - K)$. 
\end{proof}

Algorithm~\ref{alg:simple_splitting_FPE_introduction} provides a constructive way to build instances of \eqref{eq:algorithm_1_nonpar} simply by choosing 4 matrices, which reduce to only 2 in the case of pure resolvent splittings, from a beginning of 10. Each of these matrices have specific roles:

\begin{enumerate}[label=(\roman*)]
	\item $M\in\R^{n\times (n-1)}$ defines how the information stored in $\bz^k$ from previous iterations affect the input to the evaluation of each resolvent. Moreover, $M^T$ decides how the resolvent evaluations affect what is stored to the next iteration in $\bz^{k+1}$. Therefore, as evident from \eqref{eq:algorithm_1_nonpar}, both $M$ and $M^T$ are evaluated in each iteration. In Section~\ref{sec:lifted_algorithm_formulation}, we show that at the cost of adding one additional variable, we can in fact find a method that only uses $MM^T$, which is particularly useful for distributed optimization and performance enhancements.
	
	\item The matrix $P\in\R^{n\times(n-1)}$ does not enter in the algorithm, but is there to enforce that $S\succeq MM^T+\tfrac{1}{2}(H-K^T)\diag(\beta)(H^T-K)$ and that $\Ran(\one)\subset \Null(S)$. It plays the role of a meta-step size since the larger $P$ is, the larger $S$ is and the smaller the step sizes in $\gamma$ become. $P$ is included to capture all methods satisfying \refallprop, but in practice, we suggest picking $P=0$.
	
	\item \( H\in\R^{n\times m} \) and \( K\in\R^{m\times n} \) define how the cocoercive terms enter into the algorithm. Specifically, for each \( j \in \llbracket1, m\rrbracket \), the \( j \)th row of \( K \) assigns averaging weights to combine the outputs of previously computed resolvent evaluations within the iteration to form the input to $C_j$. Similarly, the \( j \)th column of \( H \) provides averaging weights that split the \( C_j \) output among inputs to resolvent evaluations that are yet to be performed within the iteration. In \cref{sec:numerics} we present heuristics for selecting $H$ and $K$ to achieve excellent performances. 
\end{enumerate}

\subsection{Lifted algorithm reformulation}\label{sec:lifted_algorithm_formulation}

Through the change of variables $\bw^k:=M\bz^k$ in \eqref{eq:algorithm_1_nonpar}, one can directly select $MM^T$ and evaluate it in the algorithm instead of selecting $M$ and evaluating $M$ and $M^T$. Denoted by $\mathcal{L}:=MM^T$, the resulting method reads as:
\begin{equation}\label{eq:alg_MM_lifted}
	\left\{\begin{aligned}
		&\bx^{k+1} = J_{\Gamma \bA}(\Gamma L\bx^{k+1} - \Gamma H\bC(K\bx^{k+1}) + \Gamma \bw^k ),\\
		&\bw^{k+1} = \bw^k-\theta \cL\bx^{k+1},
	\end{aligned}\right. 
\end{equation}
and generates a weakly converging sequence provided that $\bw^0 \in \cH^n$ is such that $\one^T \bw^0=0$, e.g., $\bw^0=0$. This equivalent algorithm may reduce the computational cost as:
\begin{enumerate}
	\item It may be cheaper to evaluate $\cL$ as one matrix than $M$ and $M^T$ individually. 
	\item By picking $\mathcal{L}$ as the graph laplacian of a connected graph, we can easily control its sparsity pattern, allowing simple distributed protocols, cf.~Section~\ref{sec:special_cases_adapted_forward_backward}.
	\item We can directly set:
	\begin{equation}\label{eq:define_lap}
		\mathcal{L} := n I_n - \one\one^T,
	\end{equation}
	i.e., the graph laplacian of a fully connected graph, which provides great performance in practice, cf.~Section~\ref{sec:numerics}.  
\end{enumerate}
However, the resulting algorithm stores $\bw^k \in \cH^n$ between iterations, thereby sacrificing the minimal lifting property. Of course, recovering $M$ from $\mathcal{L}$ is feasible, but requires a spectral decomposition, which can be unstable and expensive if $n$ is large. This why we recommend using this formulation if possible. Note that there is no loss of generality in doing so since the decomposition $\mathcal{L}=MM^T$ with $M$ full-rank is unique modulo orthogonal transformation and all corresponding algorithms with minimal lifting are equivalent in the sense of Definition~\ref{def:equivalence}, see \cite[Remark 3.4]{bcn24}.

\subsection{Special cases}\label{sec:special_cases}

In this section, we collect some relevant special cases of Algorithm~\ref{alg:simple_splitting_FPE_introduction}, and provide a new one with important practical motivations.
 
\subsubsection{Douglas--Rachford and Davis--Yin}
Let us start with the case $n=2$, and $m\geq 0$. For fixed-point encoding and causality according to Section~\ref{sec:causality}, we must have the matrices:
\begin{equation}
	K^T = \begin{bmatrix}
		1 & 1 & \dots & 1\\
		0 & 0 & \dots & 0
	\end{bmatrix},
	\quad 
	H = \begin{bmatrix}
		0 & 0 & \dots & 0\\
		1 & 1 & \dots & 1
	\end{bmatrix}, \quad \text{and} \quad 
		L = \begin{bmatrix}
			0 & 0\\
			\frac{1}{\gamma_1} + \frac{1}{\gamma_2} & 0
		\end{bmatrix}
\end{equation}
for $\Gamma = \diag(\gamma_1, \gamma_2)$ since $K\one = H^T\one = \one$ and $\one^T(\Gamma^{-1}-L)\one = 0$. Moreover, due to the nullspace condition on $M$ we must have
$M^T = \lambda \begin{bmatrix}
	1 & -1
\end{bmatrix}$ for some $\lambda \in \real$. Let us now denote by $\hat \beta:=\beta_1 + \dots + \beta_m$ and $\hat \gamma_i := \gamma_i^{-1}$ for $i=1, 2$. To satisfy Theorem~\ref{thm:frugal_splitting_iff}\ref{thm:splitting_nonexp_iff_4} we must have
\begin{equation}\label{eq:condition_davis_yin_1}
	\begin{aligned}
		0&\preceq  2\Gamma^{-1}-L-L^T-MM^T-\frac{1}{2}\left(H-K^T\right)\diag(\beta)(H^T-K) = \\
		&= \begin{bmatrix}
			2 \hat \gamma_1 & -(\hat \gamma_1 + \hat \gamma_2)\\
			-(\hat \gamma_1 + \hat \gamma_2) & 2 \hat \gamma_2
		\end{bmatrix} -  \lambda^2 \begin{bmatrix} 1 & -1\\
			-1 & 1\end{bmatrix} - \frac{\beta}{2}\begin{bmatrix}
			1 & -1\\
			-1 & 1
		\end{bmatrix}.
	\end{aligned}
\end{equation}
Let us assume first that $\hat \beta = 0$. Then, \eqref{eq:condition_davis_yin_1} implies $\lambda^2 \leq 2 \min\{\hat \gamma_1, \hat \gamma_2\}$ and
\begin{equation}
	 (2 \hat \gamma_1 - \lambda^2)(2 \hat \gamma_2 - \lambda^2)- (\hat \gamma_1 - \hat \gamma_2 - \lambda^2)^2 = 0,
\end{equation}
which, after elementary calculations, yields $\hat \gamma_1 = \hat \gamma_2$, i.e., $\gamma_1 = \gamma_2 = \gamma >0$. Let us now consider once again \eqref{eq:condition_davis_yin_1} with $\hat \beta > 0$. This condition now yields $\frac{2}{\gamma}  - \lambda^2 - \frac{\beta}{2}\geq 0$, i.e., $\frac{4-\beta \gamma}{2}\geq \lambda^2 \gamma$. Finally, applying the substitution $\bar z := \gamma \lambda z$ and denoting by $\bar \theta := \lambda^2 \gamma$, we obtain:
\begin{equation}
	\left\{
	\begin{aligned}
		&x_1^{k+1} = J_{\gamma A_1} (\bar z^k),\\
		&x_2^{k+1} = J_{\gamma A_2}(2 x_1^{k+1} - \gamma(C_1(x_1^{k+1}) + \dots + C_m(x_1^{k+1})) - \bar z^k),\\
		&\bar z^{k+1} = \bar z^k - \bar \theta (x_2^{k+1} - x_1^{k+1}),
	\end{aligned}\right.
\end{equation}
which is averaged if and only if $\bar \theta>0$ a $\gamma>0$ are chosen to satisfy $\frac{4-\beta \gamma}{2}\geq  \bar \theta >0$. These conditions, which in particular allow for \emph{twice larger step sizes} than in the original paper \cite[Theorem 2.1]{dy17}, i.e., $\gamma \leq \frac{4}{\beta}$ (but at the cost of decreasing $\bar \theta$), were obtained in \cite[Section 3]{bcln22} and later in \cite{at22}. Here, they follow immediately from Theorem~\ref{thm:frugal_splitting_iff}.

\subsubsection{Graph-Douglas--Rachford}\label{sec:special_cases_graph_drs}
We now consider the case $n >2$, but let us start first with $m=0$. In this case, we only need to pick $M$ and $P$. The choice of two matrices with a specified kernel reflects in \cite{bcn24} with choice of two directed connected graphs\footnote{Note that we inverted the role of $\mathcal{G}$ and $\mathcal{G}'$ in relation to \cite{bcn24}.} on (the same) $n$ nodes, $\mathcal{G}:=(\cN, \cE)$ and $\mathcal{G}':=(\cN, \cE')$ with $\cE\subset \cE'$ and the property
\begin{equation}\label{eq:graph_property}
(i, j)\in \cE \implies i<j, \quad \text{and} \quad (i, j)\in \cE' \implies i<j.
\end{equation}
Then, the matrices $M$ and $P$ are defined as follows:
\begin{enumerate}
	\item\label{item:point_1_graphdrs} $M$ is obtained as a full rank factorization of the graph laplacian of $\mathcal{G}$ as $\mathcal{L}=MM^T$. Note that $\Null (M^T) = \Ran (\one)$ follows by construction since $\Null (\mathcal{L})=\Ran (\one)$ by definition of graph laplacian, and the fact that $\mathcal{G}$ is connected.
	\item $P \in \R^{n \times (n-1)}$ is any matrix such that $\overline{\mathcal{L}}=PP^T$, where $\overline{\mathcal{L}}$ is the graph laplacian of $\mathcal{G}'\setminus \mathcal{G}:=(\cN, \cE'\setminus \cE)$. In this case, since $\overline{\cL}$ is not necessarily connected, we only have $\Ran(\one)\subset \Null(\overline{\cL})$. Hence, $P$ might not be full rank, but always satisfies $\Ran(\one)\subset \Null(P^T)$.
\end{enumerate}
With these choices, the order of evaluations of the resolvents is prescribed by $\cG'$ since
\begin{equation}
	S=MM^T+ PP^T = \cL + \overline{\cL} = \cL', \quad \text{i.e., the graph laplacian of $\mathcal{L}'$}.
\end{equation}
 The method is shown to encompass several frugal resolvent methods with minimal lifting such as Ryu's three-operator splitting \cite{Ryu}, Malitsky--Tam splitting \cite{mt23}, two \emph{parallel} extensions of DRS \cite{ckch23, RyuYin_parallel}, the \emph{sequential} DRS \cite{bcln22}, among other methods. The possibility to consider general matrices $M$ and $P$ satisfying the assumptions in Algorithm~\ref{alg:simple_splitting_FPE_introduction} is mentioned in \cite[Remark 2.7]{bcn24}, and discussed in detail in the thesis \cite{chenchene_thesis}.

\subsubsection{Forward-backward devised by graphs}\label{sec:special_cases_graph_forward_backward} In the recent work \cite{acl24}, Arag\'{o}n Artacho, Campoy and L\'{o}pez-Pastor, provide an adaptation of graph-DRS in Section~\ref{sec:special_cases_graph_drs} to solve instances of \eqref{eq:Inclusion} with $m=n-1$ and $\beta_1=\dots=\beta_{n-1}=\beta$, referred to as forward-backward devised by graphs (GFB). The method requires choosing three graphs $\mathcal{G}$, $\mathcal{G}'$ and $\mathcal{G}_f$ on the same set of nodes, with $\cE \subset \cE'$, $\cE_f \subset \cE'$, and satisfying \eqref{eq:graph_property}. The graphs $\cG$ and $\cG'$ play a similar role to those of graph-DRS, while $\cG_f$ handles the evaluation of the forward terms, and must satisfy:
\begin{equation}
	(i, j) \in \cE_f \ \text{and} \ (i', j) \in \cE_f \implies i = i'.
\end{equation}
The GFB method can be obtained as a special case of Algorithm~\ref{alg:simple_splitting_FPE_introduction} by picking: $M$ as for the graph-DRS method, see Point~\ref{item:point_1_graphdrs} in Section~\ref{sec:special_cases_graph_drs}. Then, $H$ and $K$ such that
\begin{equation}
	H_{ij} = \begin{cases}
		1 & \text{if} \ i = j - 1,\\
		0 & \text{else},
	\end{cases} \quad \text{and} \quad
	K_{ij} = 
	\begin{cases}
		1 & \text{if} \ (j, i) \in \cE_f,\\
		0 & \text{else}.
	\end{cases}
\end{equation}
Finally, $P$ is picked in such a way that $S=\cL + \overline{\cL} + \frac{\beta}{2} \cL' =(1 + \frac{\beta}{2})\cL'$. A simple calculation shows that $\frac{1}{2}(H-K^T)\diag(\beta) (H^T - K) = \frac{\beta}{2}\cL_f$, where $\cL_f$ is the graph laplacian of $\cG_f$. Therefore, the matrix $P$ in this case is such that $PP^T=\overline{\cL} + \frac{\beta}{2} \cL' - \frac{\beta}{2}\cL_f = \overline{\cL} + \frac{\beta}{2} (\cL' - \cL_f)$. This shows that even if $\cG=\cG'=\cG_f$ the matrix $P$ need not be necessarily zero, which, as we will see in Section~\ref{sec:num_testing_P}, has a negative influence on performances

As shown in \cite[Section 4.2]{acl24}, this method encompasses several existing methods in the literature. Special cases include the \emph{ring forward-backward} by Arag\'{o}n-Artacho, Malitsky, Tam and Torregrosa-Belén \cite{amtt23}, the \emph{sequential} and \emph{parallel} Davis--Yin by two of the authors and Bredies, Lorenz \cite{bcln22}, and the so-called \emph{four-operator splitting} by Zong, Tang, and Zhang \cite{ztz23}.

\subsubsection{Adapted graph forward-backward}\label{sec:special_cases_adapted_forward_backward} In this section, we propose yet another variant of Algorithm~\ref{alg:simple_splitting_FPE_introduction}, this time with $m=(n-1)\frac{n}{2}$. We pick once again $M$ such that $\cL=MM^T$ is the graph laplacian of a graph $\cG$ satisfying \eqref{eq:graph_property}, and set $P=0$. To handle the forward terms, we suppose that inside the $i$th resolvent (for $i > 1$), we evaluate at most $i-1$ forward terms---renamed $C_{i,1}, \dots, C_{i, i-1}$ in this setting---on the corresponding $x_h^{k+1}$ such that $(h, i)\in \cE$. These rules univocally define $H$ and $K$, which we do not display explicitly for the sake of space\footnote{Note that thay can be easily deduced from Algorithm~\ref{alg:distributed_algorithm}.}. These choices yield $W:=\frac{1}{2}(H-K^T)\diag(\beta)(H^T - K)$ such that
\begin{equation}
	\bx^T W \bx = \frac{1}{2}\sum_{i=2}^n\sum_{(h, i)\in \cE} \beta_{ih}\|x_i - x_h\|^2,\quad \text{for all} \ \bx = (x_1,\dots, x_N)\in \R^n.
\end{equation}
In particular, denoted by $d_i$ the degree of node $i$ in $\cG$, the components of $\gamma = 2\diag(S)^{-1}$ and $S=\cL + W$ read as: 
\begin{equation}
	\gamma_i := \frac{2}{d_i + \frac{1}{2}\left(\sum_{(h, i)\in \cE}\beta_{ih} + \sum_{(i, j)\in \cE} \beta_{ji}\right)}, \quad \text{and} \quad S_{ih} = 
	\begin{cases}
		- 1 - \beta_{ih} & \text{if} \ (h, i) \in \cE,\\
		0 & \text{else}.
	\end{cases} 
\end{equation}
This setting is particularly interesting as it yields Algorithm~\ref{alg:distributed_algorithm} (displayed in the lifted variant according to Section~\ref{sec:lifted_algorithm_formulation}), which can be implemented in a fully decentralized distributed fashion without the need of a global cocoercivity constant. It thus extends both \cite{acl24} and \cite{naldi_thesis}.

\begin{remark}
	While we refrain from displaying the distributed protocol for Algorithm~\ref{alg:distributed_algorithm}, which can be obtained by adapting, e.g., \cite[Algorithm 4.2]{bcn24} to this setting, we would like to highlight some of its key features:
	\begin{enumerate}
		\item Each agent possesses its local variable $w_i^k \in \cH$, and each operator $C_{i,h}$ for all $(h, i)\in \cE$.
		\item To update $x_i^{k+1}$, agent $i$ only requires the knowledge of: $\beta_{ih}$ and $\beta_{ji}$ for all $(h, i)\in \cE$ and $(i, j)\in \cE$, which can be obtained by the neighbors, as well as the knowledge of $x_h^{k+1}$ for all $(h, i)\in \cE$ which can also obtained by its $h$ neighbor.
		\item To update $z_i^{k+1}$ the method requires a second round of communication, as not only $x_h^{k+1}$ with $(h, i)\in \cE$ are needed, but also all $x_j^{k+1}$ for $(i, j)\in \cE$, i.e., $j>i$. This complies with \cite[Algorithm 4.2]{bcn24}.
		\item Agent $i$ updates $w_i^{k+1}$ following the direction $\frac{1}{d_i}\sum_{i\sim j}x_j^{k+1} - x_i^{k+1}$, i.e., it averages the solution estimates provided by its neighbors and takes a step in that direction---the higher its degree, the more it trusts the average opinion of its neighbors.
	\end{enumerate}
	Observe as well in Algorithm~\ref{alg:distributed_algorithm} the effect of the heterogeneity of data: Only the agents that are adjacent to agents that evaluate forward steps have $\hat \gamma_i > d_i$, i.e., smaller local steps, while all the others have $\hat \gamma_i= d_i$.
\end{remark}

\begin{algorithm}[t]
	\caption{Adapted Graph Forward-Backward Splitting (aGFB).}\label{alg:distributed_algorithm}
	\textbf{Pick:} $\cG=(\cN, \cE)$ a connected directed graph satisfying \eqref{eq:graph_property}, and $\theta \in (0, 1)$\\
	\textbf{Set:} $\hat \gamma_i:= d_i + \frac{1}{2}(\sum_{(h, i)\in \cE} \beta_{ih} + \sum_{(i, j)\in \cE} \beta_{ji})$ for $i \in \llbracket1, n\rrbracket$
	
	\textbf{Input:} {$\bz^0=0\in \cH^{n}$}
	
	\For {$k = 0, 1, 2, \ldots$}{
		\For {$i = 1, 2, \dots, n$}{
			$\displaystyle x_i^{k+1} = J_{\frac{2}{\hat \gamma_i}A_i}\bigg(\frac{2}{\hat \gamma_i}\sum_{(h, i)\in \cE} (1+\beta_{ih}) x_h^{k+1} - \frac{2}{\hat \gamma_i} \sum_{(h, i) \in \cE} C_{i,h} \left(x_h^{k+1}\right)+\frac{2}{\hat \gamma_i} w_i^k \bigg)$
		}
		\For {$i = 1, 2, \dots, n$}{
			$\displaystyle w_i^{k+1} = w_i^k + d_i\theta \bigg(\frac{1}{d_i}\sum_{i \sim j}x_j^{k+1} - x_i^{k+1}\bigg)$ \hfill \textcolor{gray}{$\rhd$ the symbol $\sim$ means adjacent}
		}
	}
\end{algorithm}

\section{Numerical Experiments}\label{sec:numerics}

In this section, we present our numerical experiments, which are performed in Python on a 12thGen.~Intel(R) Core(TM) i7–1255U, 1.70–4.70 GHz laptop with 16 Gb of RAM and are available for reproducibility at \href{https://github.com/echnen/split-forward-backward}{https://github.com/echnen/split-forward-backward}.

\subsection{Investigating algorithm's parameters}\label{sec:num_investigating_algorithm_parameters}
In this section, we present a series of numerical experiments to investigate the influence of the matrices $M$, $P$, and $H$, $K$ on the performance of Algorithm~\ref{alg:simple_splitting_FPE_introduction}, and we provide suggestion for their design to achieve good performance.

\subsubsection{Data preparation}\label{sec:num_data_preparation}

For $p, d, n\in \N$, $0\leq\delta_1\leq \delta_2$, a matrix $\Psi \in \R^{p, d}$, a vector $y \in \R^p$, a sample of points $\{\xi_i\}_{i=1}^n\subset \R^d$, we consider the following convex optimization problem:
\begin{equation}\label{eq:num_objective_function}
    \min_{x \in \R^d} \ \|x - \xi_1\|+\dots + \|x-\xi_n\| + H_{\delta_1, \delta_2}(\Psi x - y),
\end{equation}
where $H_{\delta_1, \delta_2}: \R^p\to \R$ is a Huber-like smooth function defined for all $z:=(z_1,\dots, z_p) \in \R^p$ by:
\begin{equation*}
    H_{\delta_1, \delta_2}(z):=\sum_{i=1}^p h_{\delta_1, \delta_2}(z_i), \quad \quad h_{\delta_1, \delta_2}(z_i):=\begin{cases}
        0 & \text{if} \ |z_i|\leq \delta_1,\\
        \frac{1}{2}(z_i- \delta_1)^2 & \text{if} \ |z_i| \in [\delta_1, \delta_2],\\
        (\delta_2- \delta_1)|z_i|-\frac{1}{2}(\delta_2^2 - \delta_1^2) & \text{else}.
    \end{cases}
\end{equation*}
The function $h_{\delta_1, \delta_2}$ is differentiable with $1$-Lipschitz continuous gradient and thus $H_{\delta_1, \delta_2}$ has $\|\Psi\|_2$-Lipschitz continuous gradient. Instead of computing or estimating the norm of $\Psi$, we can split the forward term into $m\leq p$ terms. Doing so, since $h_{\delta_1, \delta_2}$ is $1$-Lipschitz, we can take, for all $i \in \llbracket1, m\rrbracket$, $\beta_i := \|\Psi_{I_i}\Psi_{I_i}^T\|_2$ where $\Psi_{I_i}\in \R^{|I_i|\times d}$ denotes the $i$th block formed with the rows of $\Psi$ indexed by $I_i\subset \llbracket1, p\rrbracket$. Note that if $I_i=\{i\}$ and $m=p$, then $C_i(x) := h_{\delta_1, \delta_2}'(\Psi_{i} x - y_{i})\Psi_i^T$ where $\Psi_i$ is the $i^{th}$ row of $\Psi$ and $\beta_i = \|\Psi_i\|^2$, which can be computed easily. Regarding the nonsmooth terms, we consider $A_i := \partial g_i$ with $g_i(x):=\|x - \xi_i\|$ for all $i\in \llbracket1, n\rrbracket$. Observe that for all $\tau>0$, $J_{\tau A_i}$ coincides with $\prox_{\tau g_i}$, which admits a simple closed-form expression via a standard soft-thresholding.

\subsubsection{Algorithm design}\label{sec:num_randomized_algorithm_design}

To capture the generality of Algorithm~\ref{alg:simple_splitting_FPE_introduction}, with the aim of systematically testing the influence of the parametrization matrices, we follow construction principles that allow us to randomly generate instances of $M, H, K,$ and $P$ guaranteeing that all the hypotheses are met. This setup enables extensive numerical experiments and leads us to three heuristics that consistently deliver strong performances. In the following, it will be convenient to denote by $\mathcal{U}(I)$, the uniform distribution on the interval $I\subset \R$.\medskip

\paragraph{\textit{Defining $M$ and $\mathcal{L}:=MM^T$}} We can easily define a matrix $M$ satisfying the properties in Algorithm~\ref{alg:simple_splitting_FPE_introduction} by setting
\begin{equation}\label{eq:define_M}
    M:=\left(I_n - \tfrac{1}{n}\one\one^T\right) \hat{M}, \quad \text{for} \ \hat{M} \ \text{sampled from} \ \mathcal{U}(I_M)^{n\times (n-1)}, \ I_M \subset \R.
\end{equation}
This yields a full rank matrix with probability $1$, and allows us to compare different randomized choices of $M$, cf.~Section \ref{sec:num_testing_M}. In practice, we opt to implement the equivalent formulation of Algorithm \ref{alg:simple_splitting_FPE_introduction} described in Section~\ref{sec:lifted_algorithm_formulation}, at the cost of sacrificing minimality. This allows us to pick $\mathcal{L}$ as in \eqref{eq:define_lap}. This matrix is the one that maximizes the magnitude of the smallest nonzero eigenvalue among all matrices with fixed spectral norm of the form $MM^T$ with $M$ as in \eqref{eq:define_M}. This choice, already suggested, e.g., in \cite{bcn24}, yields our first heuristic:

\begin{heuristic}\label{heuristic:M}
	In Algorithm~\ref{alg:simple_splitting_FPE_introduction}, pick $M\in \R^{n\times (n-1)}$ such that $\cL=MM^T=nI_n - \one \one^T$. To avoid finding the corresponding $M$, implement the equivalent version \eqref{eq:alg_MM_lifted} when possible.
\end{heuristic}

\medskip

\paragraph{\textit{Defining $P$}} 
To define a matrix $P$ with the properties in Algorithm~\ref{alg:simple_splitting_FPE_introduction}, we can sample $\hat{M}\in\R^{n \times m_1}$ for any $m_1\leq n-1$, and set $P$ as in \eqref{eq:define_M}. Note that, this time, any such an instance is feasible, even $P=0$, no matter the rank condition. Since a smaller $S$ gives larger resolvent step sizes, we suggest to select $P$ such that $PP^T$ is as small as possible, which yields our second heuristic:

\begin{heuristic}\label{heuristic:P}
	In Algorithm~\ref{alg:simple_splitting_FPE_introduction}, pick $P=0$.
\end{heuristic}

\medskip

\paragraph{\textit{Defining causal $H$ and $K$}}\label{sec:defining_causal_HK} To define $H, K^T \in \R^{n\times m}$, we first sample $\hat H$ and $\hat K^T$ from $\mathcal{U}(I_H)^{n\times m}$ and $\mathcal{U}(I_K)^{n\times m}$ on the intervals $I_H$ and $I_K$, respectively. We generate randomly an $(m, n)$-nondecreasing vector $F:=(F_1,\dots, F_m)$ (cf.~Definition~\ref{def:nondecreasing-index-set}) and, for all $i \in \llbracket1, n\rrbracket$, we set
\begin{equation}
    \hat H_{i j} = 0, \ \hat K_{hi}=0, \quad \text{for all} \ h \leq F_i \leq j.
\end{equation}
Eventually, we normalize the sum to one as: $H_{ij}:= \hat{H}_{ij} / \sum_{h=1}^n \hat H_{hj}$ and $K_{ji}:= \hat{K}_{ji} / \sum_{h=1}^n \hat K_{ih}$ for all $i\in \llbracket1, n\rrbracket$ and $j \in \llbracket1, m\rrbracket$. With probability $1$, this procedure yields a couple of causal matrices $H$ and $K$ such that $H^T\one=K\one=\one$. For further details, we refer the reader to Example~\ref{example:causal_matrices}.

The $H$ and $K$ matrices appear in the convergence condition in Item~\ref{thm:splitting_nonexp_iff_4} of Theorem~\ref{thm:frugal_splitting_iff}. Since in the matrix $S$ there is no distinction between $PP^T$ and $W$, to remain consistent with Heuristic~\ref{heuristic:P}, we propose the following:
\begin{heuristic}\label{heuristic:H_and_K}
	In Algorithm~\ref{alg:simple_splitting_FPE_introduction}, pick $H$ and $K$ that are causal in the sense of Definition~\ref{def:causal_pair} and minimize $\|W\|_2$, i.e., that solve the following convex optimization problem:
	\begin{equation}
		\label{eq:Optimization_KH}
		\begin{aligned}     
			&\underset{K,H^T\in \real^{n\times m}}{\text{minimize}} && \left\Vert\sqrt{\diag(\beta)}(K-H^T)\right\Vert_2,\\
			&\text{subject to } &&  K\one = H^T\one = \one,\\
			&&& K\in \mathcal{S}_{n,m}(F), \  H^T\in \mathcal{S}^c_{n,m}(F),
		\end{aligned}
	\end{equation}
where the sets $\mathcal{S}_{n,m}(F)$ and $\mathcal{S}^c_{n,m}(F)$ are defined in Definition~\ref{def:staircase_structure} and $F$ is a fixed $(m, n)$-nondecreasing vector according to Definition~\ref{def:nondecreasing-index-set}.
\end{heuristic}

\subsubsection{The influence of $M$}\label{sec:num_testing_M} In the following experiment, we test the influence of $M$. Specifically, we consider a problem instance as detailed in Section~\ref{sec:num_data_preparation}, with $n=20$. We set $P=0$, generate $H, K$ randomly as detailed in Section~\ref{sec:num_randomized_algorithm_design}, and run Algorithm~\ref{alg:simple_splitting_FPE_introduction} $200$ times for $100$ iterations with $\mathcal{L}$ either: generated as $\mathcal{L}=MM^T$ with $M$ as specified in \eqref{eq:define_M} (only $20$ times); chosen as specified in \eqref{eq:define_lap}; or as the graph laplacian of a connected Watts--Strogatz small-world graph with a random degree picked from $\llbracket 1, n\rrbracket$. Additionally, we normalize all the norms to be $\|\mathcal{L}\|_2 = 1$.

We record the last objective function residual and variance and the corresponding \emph{algebraic connectivity}, i.e., the smallest nonzero eigenvalue of $\mathcal{L}$ (which for graph laplacians is a measure of connectivity of the graph and is usually denoted by $\lambda_1$), and show the results in Figure~\ref{fig:experiment_lap}. Figure~\ref{fig:experiment_lap} confirms that also in the forward-backward case of the present paper, the choice \eqref{eq:define_lap} is the one to be preferred if no distributed implementations are needed, confirming Heuristic~\ref{heuristic:M}. It also suggests that graph laplacians always yield better performance than randomly sampled matrices as in \eqref{eq:define_M}.

\subsubsection{The influence of $P$}\label{sec:num_testing_P} In this experiment, we test the influence of the matrix $P$ complementing a first insight in \cite[Section 5.1]{bcn24}. Specifically, we consider the same setting as in Section~\ref{sec:num_testing_M} but this time we fix a (randomly generated) laplacian matrix $\mathcal{L}$, and each time i) we generate a matrix $P$ randomly as described in Section \ref{sec:num_randomized_algorithm_design}, we compute $PP^T$, and ii) we rescale $PP^T$ in such a way that $\|PP^T\|_2$ matches a random number sampled from $[0, 1]$. We record the last objective function residual and the value of $\|PP^T\|_2$ and show the results in Figure~\ref{fig:experiment_P}.

\Cref{fig:experiment_P} confirms that setting \( P \neq 0 \) negatively impacts the convergence of Algorithm~\ref{alg:simple_splitting_FPE_introduction}, and suggests picking \( P = 0 \), as suggested in Heuristics~\ref{heuristic:P}. The same trend will persist in the comparison with existing algorithms in Section~\ref{sec:num_comparison}, where some are formulated with \( P \neq 0 \).

\begin{figure}[t]
    \centering
    \begin{subfigure}{0.4\linewidth}
        \includegraphics[width=\linewidth]{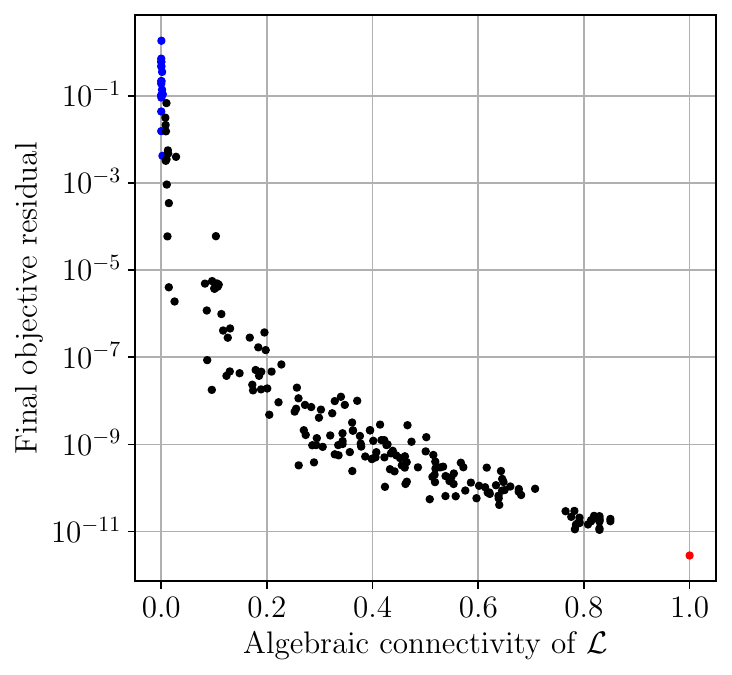}
        \caption{Testing the influence of the spectrum of $\mathcal{L}$: Final residual against $\lambda_1$.}
    \label{fig:experiment_lap}
    \end{subfigure}\hfill
    \begin{subfigure}{0.4\linewidth}
        \includegraphics[width=\linewidth]{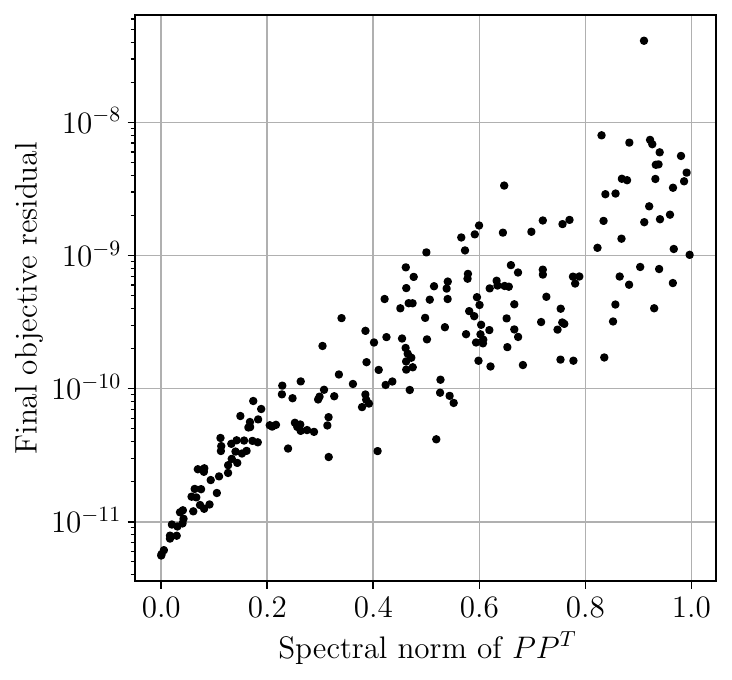}
        \caption{Testing the influence of $P$: Final residual against $\|PP^T\|_2$.}
    \label{fig:experiment_P}
    \end{subfigure}
    \caption{How to choose the matrices $\mathcal{L}$ and $P$: Results of the experiments in Section~\ref{sec:num_testing_M} and Section~\ref{sec:num_testing_P}. In Figure \ref{fig:experiment_lap}, the blue dots correspond to $\mathcal{L}$ generated as $\mathcal{L}=MM^T$ with $M$ as in \eqref{eq:define_M}, the red dot to the case with $\mathcal{L}$ as in \eqref{eq:define_lap} and the black ones to those cases in which $\mathcal{L}$ is a graph laplacian. }
\end{figure}

\subsubsection{The influence of $H$ and $K$}\label{sec:num_studying_H_and_K} In the following set of experiments, we study the choices of $H$ and $K$, which prescribe how the forward terms are handled.

We first show that having individual estimates of the cocoercivity constants of each forward term $C_i$ can significantly enhance the performance of the method. To do so, we consider two specific instances of problem \eqref{eq:num_objective_function} that differ only in the choice of $\Psi\in \R^{p \times d}$. First, we sample it uniformly, then, we consider the same matrix but scaling $2$ random rows of $\Psi$ of a factor $5$ to simulate heterogeneity in the data fidelity term. In both cases, we fix the starting point $z = 0$ and run randomized variants of Algorithm~\ref{alg:simple_splitting_FPE_introduction} for $20$ runs, once with possibly different $\beta_i$s, and then with $(\beta_{\max}, \dots, \beta_{\max})$ where $\beta_{\max} := \max\{\beta_1,\dots, \beta_m\}$.

In each case, we measure the objective value residual along the iterations and show the results in Figure~\ref{fig:experiment_het}. It is evident that accounting for heterogeneity of data significantly enhances the performance of the method. The same phenomenon will be reflected in Section~\ref{sec:num_comparison}, where we test Algorithm~\ref{alg:simple_splitting_FPE_introduction} against other instances in the literature, which have only been introduced with $\beta$ constant.

To test the efficiency of Heuristic~\ref{heuristic:H_and_K}, we now study the performance of Algorithm~\ref{alg:simple_splitting_FPE_introduction} in relation to the spectral norm of $W$. To do so, we consider an instance of \eqref{eq:num_objective_function} with $n=15$ and solve it with Algorithm~\ref{alg:simple_splitting_FPE_introduction} splitting randomly the smooth term in \eqref{eq:num_objective_function} into $m=1,\dots, p$ separate terms. In each case, we sample $H, K$ randomly $10$ times and run Algorithm \ref{alg:simple_splitting_FPE_introduction}. In each run, we measure objective function residual along the iterations, record $\|W\|_2$, and show the results in Figure~\ref{fig:experiment_m_obj}. Each line depicts one specific case and its color denotes the magnitude of $\|W\|_2$. We also show in Figure~\ref{fig:experiment_m_sca} the scatter plot of $m$ and the corresponding $\|W\|_2$.

While the scatter plot confirms the effectiveness of Heuristic~\ref{heuristic:H_and_K}, as the best performing case is attained picking $H$ and $K$ such that the norm of $W$ is minimized (the red dot in Figure \ref{fig:experiment_m_sca} corresponds to the red line in Figure \ref{fig:experiment_m_obj}), the best strategy to split the regular term remains unclear. In particular, Figure~\ref{fig:experiment_m_sca} suggests picking $m\simeq n$. 

\begin{figure}[t]
    \centering
    \begin{subfigure}{0.32\linewidth}
        \includegraphics[width=\linewidth]{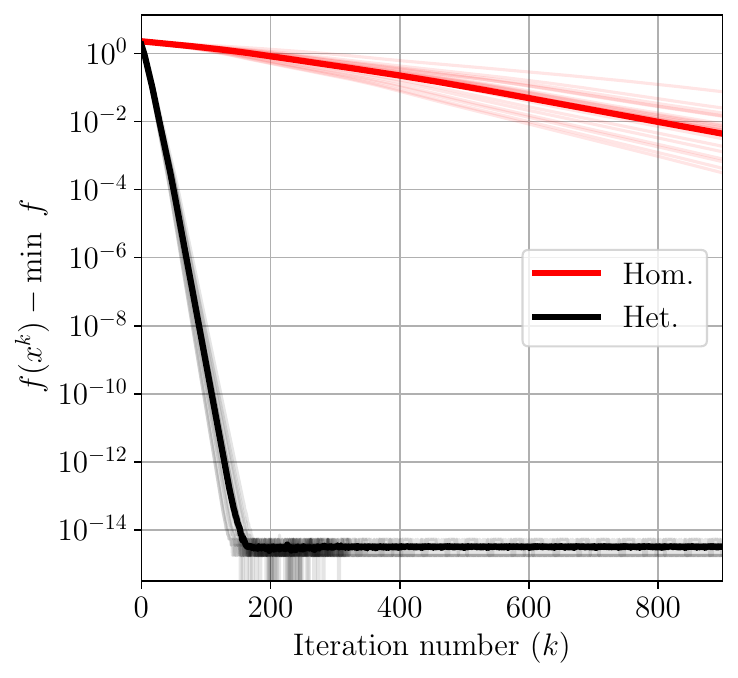}
        \caption{Accounting for heterogeneity of data enhances performances.}
        \label{fig:experiment_het}
    \end{subfigure} \hspace{0.005\linewidth}
    \begin{subfigure}{.32\linewidth}
		\includegraphics[width=\linewidth]{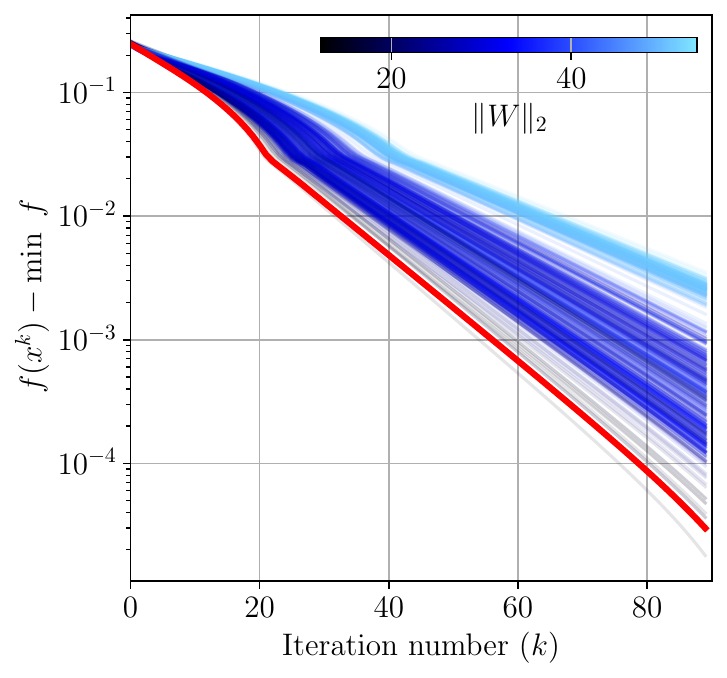}
		\caption{Objective value decrease highlighting $\|W\|_2$.}
		\label{fig:experiment_m_obj}
	\end{subfigure} \hspace{0.005\linewidth}
    \begin{subfigure}{.32\linewidth}
		\includegraphics[width=\linewidth]{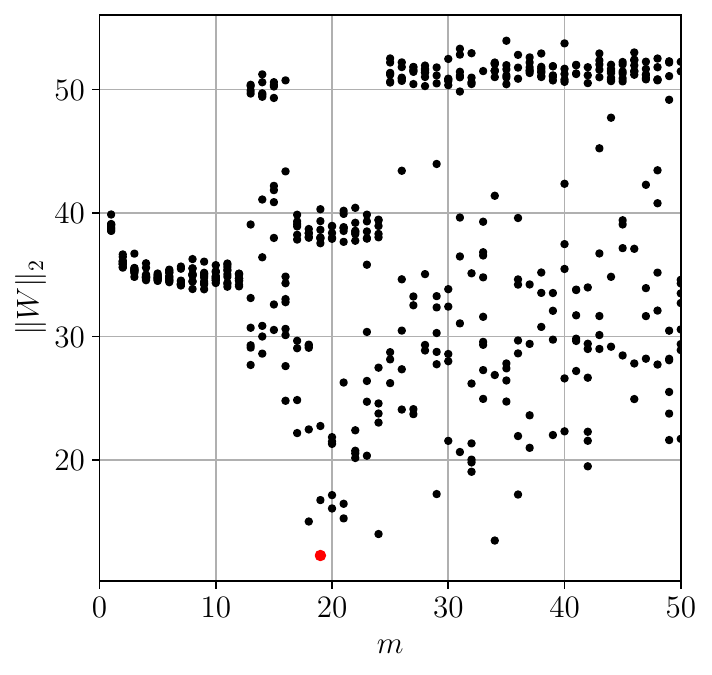}
		\caption{The values of $m$ against the corresponding values of $\|W\|_2$.}
		\label{fig:experiment_m_sca}
	\end{subfigure}
    \caption{How to choose $H$ and $K$: The benefit of accounting for heterogeneity of data (Figure \ref{fig:experiment_het}) and the influence of the spectral norm of $W$ on the performance of Algorithm~\ref{alg:simple_splitting_FPE_introduction} with different choices of $H$, $K$ (Figures \ref{fig:experiment_m_obj} and \ref{fig:experiment_m_sca}). Specifically, the red point in Figure \ref{fig:experiment_m_sca} corresponds to the red line in Figure \ref{fig:experiment_m_obj}. See Section~\ref{sec:num_studying_H_and_K} for further details.}
    \label{fig:experiment_H_K_initial}
\end{figure}

\subsection{Comparison with existing schemes}\label{sec:num_comparison}

In our second set of experiments, we consider the specific instance of Algorithm~\ref{alg:simple_splitting_FPE_introduction} introduced in Algorithm~\ref{alg:distributed_algorithm}, and its optimized variant enhanced with Heuristics~\ref{heuristic:M}, \ref{heuristic:P}, and \ref{heuristic:H_and_K}, and test it against other special cases of Algorithm \ref{alg:simple_splitting_FPE_introduction} in the literature. Specifically, we instantiate the following methods:
\begin{itemize}
    \item (aGFB) The adapted graph forward-backward method outlined in Algorithm~\ref{alg:distributed_algorithm} with $\mathcal{L}$ as in \eqref{eq:define_lap}, $P=0$, and $H, K$ described in Section~\ref{sec:special_cases_adapted_forward_backward}.
    \item (SFB+) Algorithm~\ref{alg:simple_splitting_FPE_introduction} enhanced with Heuristics~\ref{heuristic:M}, \ref{heuristic:P} and \ref{heuristic:H_and_K}, which we refer to as \emph{Split-Forward-Backward+}.
    \item (GFB) The forward-backward splitting devised by graphs introduced recently in \cite{acl24}. It can be obtained by picking $\mathcal{L}$ as in \eqref{eq:define_lap}, and $H$ and $K$ as described in Section~\ref{sec:special_cases_graph_forward_backward}.
    \item (RFB) The forward-backward splitting on ring networks introduced by Aragón-Artacho, Malitsky, Tam and Torregrosa-Belén in \cite{amtt23}.
    \item (SDY) The Sequential Davis--Yin  method by Bredies, Chenchene, Lorenz, Naldi in \cite{bcln22}.
\end{itemize}

\subsubsection{Comparison on the toy example}\label{sec:num_comparison_toy} In our first experiment, we compare the algorithms above to solve problem \eqref{eq:num_objective_function} with and without data hetereogenity. We perform a systematic comparison that randomizes the design choice of each considered algorithm ensuring fairness and unbiased-ness. In each case, we measure the objective value residual along the iterations, and show the results in Figure~\ref{fig:experiment_1}.

We can clearly see that aGFB and SFB+ perform particularly well, outperforming all the considered methods. More precisely, for homogeneous data, the performance of aGFB and SFB+ is matched only by that of SDY. In light of Section~\ref{sec:num_investigating_algorithm_parameters}, we believe the reason for this significant mismatch with GFB and CFB is that in those cases $P\neq 0$, while for aGFB, SFB+ and SDY we have $P=0$. On the other hand, for heterogeneous data, SFB+ and aGFB clearly outperform all other methods.
    
\begin{figure}[t]
    \centering
    \begin{subfigure}[b]{.32\textwidth}
    \includegraphics[width=\linewidth]{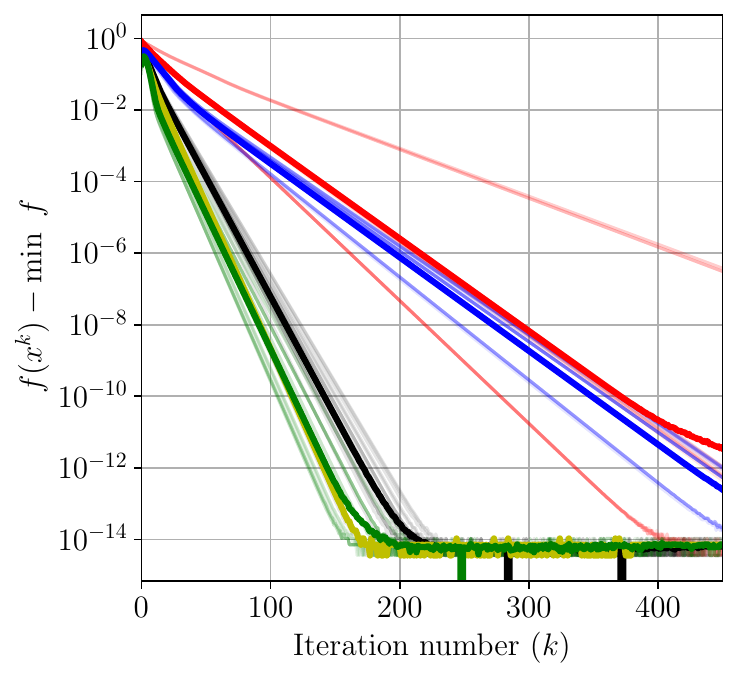}
    \centering
    \caption{Homogeneous data.}
    \label{fig:experiment_1_1}
    \end{subfigure} \hspace{0.1\linewidth}
    \begin{subfigure}[b]{.32\textwidth}
    \includegraphics[width=\linewidth]{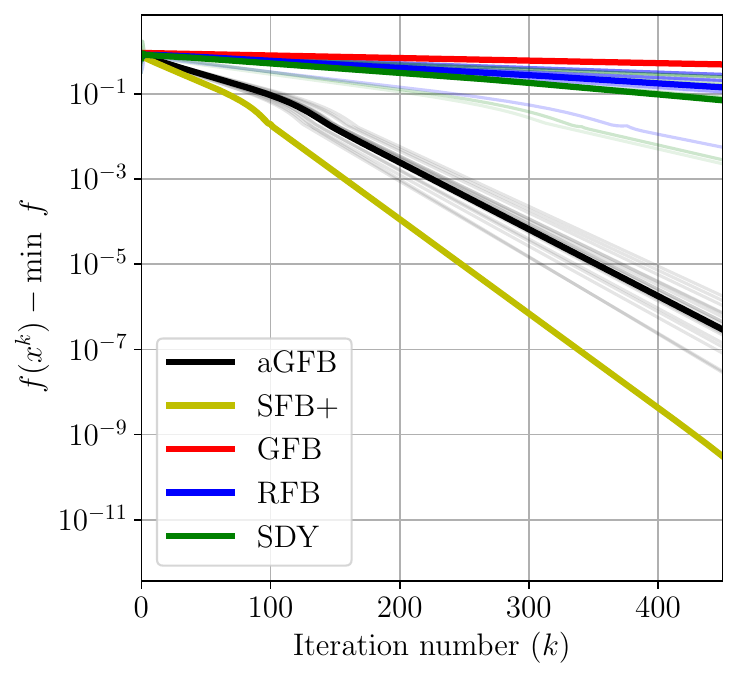}
    \caption{Heterogeneous data.}
    \label{fig:experiment_1_2}
    \end{subfigure}
    \caption{Comparison between proposed methods (aGFB, SFB+) and existing schemes to solve the toy problem in Section~\ref{sec:num_data_preparation}. See Section \ref{sec:num_comparison_toy} for details.}
    \label{fig:experiment_1}

\end{figure}

\subsubsection{Application to portfolio optimization with decarbonization}\label{sec:num_comparison_portfolio_opt} We now consider a convex portfolio optimization problem using real market data. We select \( d := 6 \) assets over the first \( p := 123 \) trading days of 2020\footnote{The data is freely available from \href{https://www.stockdata.org/login}{stockdata.org}.}. The corresponding asset returns are stored in a matrix \( r \in \mathbb{R}^{p \times d} \). Each asset $i\in \llbracket1, d\rrbracket$ has a carbon footprint encoded in three indexes $C_i^1$, $C_i^2$, $C_i^3 \in \R_+$, quantifying direct, indirect and other greenhouse gas emissions, which we take from \cite{lr22}.

The standard Markowitz portfolio optimization problem consists in finding a portfolio, i.e., $x \in \Delta$, where $\Delta \subset \R^d$ is the standard unit simplex in $\R^d$, that maximizes the expected return $\hat r^T x$ while penalizing the expected risk $x^T \hat \Sigma x$, where $\hat r$ is the vector of mean returns and $\hat \Sigma$ is the covariance matrix of asset returns. In our experiment, we consider decarbonization constraints according to the 2015 Paris agreement on climate change. These require an average of $7\%$ carbon intensity decrease each year, which, following \cite{lr22}, can be imposed as $\sum_{i=1}^d C^i x_i \leq 0.93 \sum_{i=1}^d C^i x^0_i$ with $C^i:=C^i_1 + C^i_2 + C^i_3$. To further showcase the flexibility of our algorithmic approach, we impose that the carbon intensity decreases within each scope $j$ of a percentage $\zeta_j$ by imposing the corresponding constraint $x \in \mathcal{C}_j$ for $j \in\llbracket1,3\rrbracket$. In our model problem, we further account for the so-called one-way turnover of the portfolio, i.e., an $\ell_1$ penalization with the current position $x^0$, which is standard in portfolio optimization, see, e.g., \cite{Rbbdkkns17}. Putting all together, our problem reads as:
\begin{equation}\label{eq:port_opt_objective}
\begin{aligned}
    & \underset{x \in \R^d}{\text{minimize}} && x^T \hat \Sigma x - \hat r^T x +  \|x - x^0\|_1, \\
    &\text{subject to} &&x \in \mathcal{C}_1 \cap \mathcal{C}_2 \cap \mathcal{C}_3, \ x \in \Delta.
\end{aligned}
\end{equation}

To tackle \eqref{eq:port_opt_objective}, we introduce the following non-smooth functions $g_1(x):=\|x-x^0\|_1$, $g_2(x) := \iota_\Delta(x)$, $g_3(x) = \iota_{\cC_1}(x)$, $g_4(x):= \iota_{\cC_2}(x)$ and $g_5(x):=\iota_{\cC_3}(x)$ for all $x \in \R^d$. Then, we split the smooth function $x\mapsto x^T \hat \Sigma x - \hat r^T x$ into four terms $f_i(x):=x^T \hat \Sigma_i x - \frac{1}{4}\hat r^T x$, where $\hat \Sigma_i$ is obtained by splitting $r$ evenly along the first axis into four chunks $r_i$ and setting $\hat \Sigma_i$ as being the corresponding covariance matrix. The rationale for this split is illustrated in Figure~\ref{fig:dataset}, which displays two assets' returns over the considered time frame. We can observe that the second chunk of data, corresponding to the Covid-19 crisis, exhibits markedly different behavior compared to the other periods. In light of the considerations above, isolating this segment may prove beneficial. Finally, note that the number of chunks is not restricted to four.

In summary, \eqref{eq:port_opt_objective} can be written as the sum of $n=5$ non-smooth functions with simple proximity operators, and $m=3$ smooth ones. Then, for $20$ times, we define $H$ and $K$ so as to pick one random instance for each of the considered algorithms and run Algorithm~\ref{alg:simple_splitting_FPE_introduction}. Recall that only the algorithms introduced in this paper can take different $\beta$s. For the others, we choose $\beta_{\max}:=\max\{\beta_1,\dots, \beta_m\}$. For each of these algorithm, we show in Figure~\ref{fig:experiment_portfolio} the (mean) distance to the solution (pre-computed) of the solution estimate $\{x_2^{k+1}\}_k$, which is always a feasible portfolio since it is obtained as a projection onto the simplex. In particular, we can see that our methods significantly outperform pre-existing variants, with our optimized variant SFB+ always exhibiting the best performance.

\begin{figure}[t]
	\centering
	\begin{subfigure}{.32\textwidth}
		\centering
		\includegraphics[width=\linewidth]{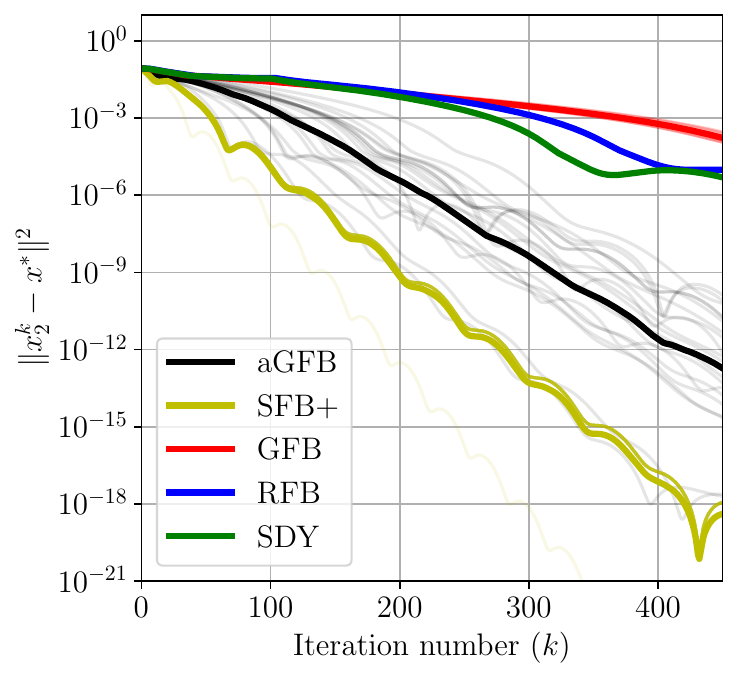}
		\caption{Distance to solution decrease.}
		\label{fig:experiment_portfolio}
	\end{subfigure}\hspace{0.1\linewidth}
	\begin{subfigure}{.35\textwidth}
		\includegraphics[width=\linewidth]{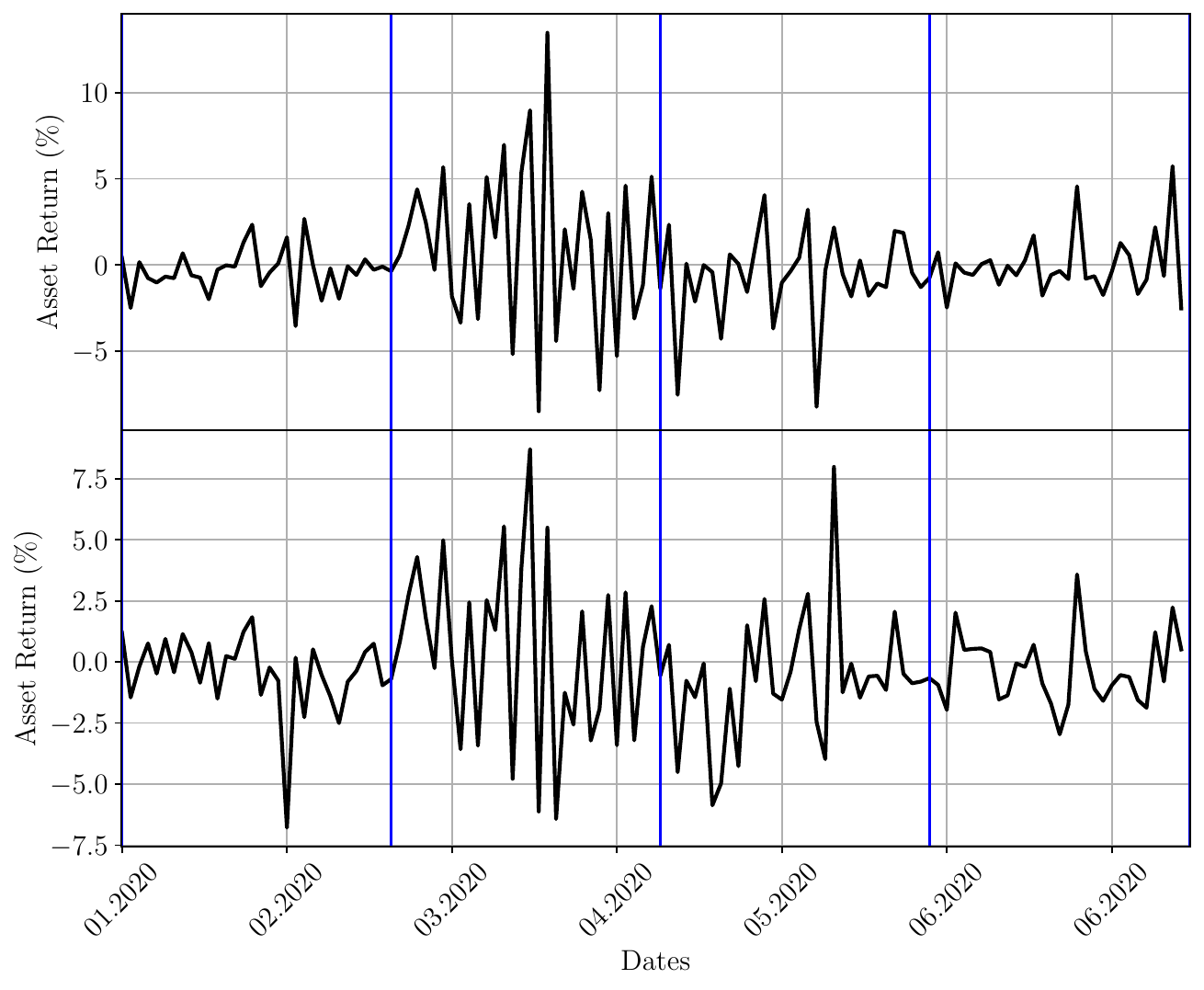}
		\caption{Two asset returns.}
		\label{fig:dataset}
	\end{subfigure}\hfill 
	\caption{Portfolio Optimization: Comparison with existing schemes and data insights. In Figure \ref{fig:dataset}, the blue vertical lines divide the data into four chunks, each defining one forward term. See Section~\ref{sec:num_comparison_portfolio_opt} for further details.}
\end{figure}

\section{Conclusions}

In this paper, we provide the first, albeit partial, result toward resolving an open problem posed by Ryu in \cite{Ryu}. We characterize all averaged frugal resolvent splittings with minimal lifting and demonstrate compelling practical applications. Whether Algorithm~\ref{alg:simple_splitting_FPE_introduction} also captures all \emph{unconditionally stable} methods remains a challenging open problem, for which we currently have no indication. Nevertheless, our results up to Remark~\ref{remark:unconditional_stability} may prove instrumental in addressing it. Our work opens further avenues for future research: We aim to adapt our method to accommodate monotone Lipschitz (but not necessarily cocoercive) operators $C_1, \dots, C_m$, and provide a full characterization of averaged splitting methods involving compositions with linear operators \emph{à la} Chambolle--Pock.

\bigskip

\noindent \textbf{Acknowledgments.}  E.N. is supported by the MUR Excellence Department Project awarded to the department of mathematics at the University of Genoa, CUP D33C23001110001, and by the US Air Force Office of Scientific Research (FA8655-22-1-7034). E.N. is a member of GNAMPA of the Istituto Nazionale di Alta Matematica (INdAM). A.\AA{} and P.G. are supported the Wallenberg AI, Autonomous Systems and Software Program (WASP) funded by the Knut and Alice Wallenberg
Foundation, and are members in the ELLIIT Strategic Research Area. P.G. acknowledge support from Vetenskapsr\aa{}det grant VR~2021-05710.

\bibliographystyle{amsplain}
\bibliography{references}

\begin{appendix}

\section{Causality}\label{sec:app_causality}

\begin{definition}[$(m,n)$-nondecreasing vector]
	\label{def:nondecreasing-index-set}
	Let $n,m\in\N$. We say that
	\[
	F=(F_1,F_2,\ldots,F_n)\in\{0,1,\ldots,m\}^n
	\] is an \emph{$(m,n)$-nondecreasing vector} if \(F_1=0\), \(F_n = m\) and $F_i\leq F_{i+1}$ for each $i\in \llbracket 1,n-1 \rrbracket$.
\end{definition}

\begin{definition}[Staircase structure]\label{def:staircase_structure}
	Let $n,m\in\N$ and $F$
	be an $(m,n)$-nondecreasing vector.
	Let $\mathcal{S}(F)$ denote the set of matrices $M\in\real^{n\times m}$ that have a \emph{staircase structure} w.r.t.~$F$, i.e., 
	\[
	M_{ij}=0 \quad \text{for all } i=1,\ldots,n \text{ and } j > F_i.
	\]
\end{definition}
Conversely, $\mathcal{S}^{c}(F)$ denotes the set of matrices $M\in\real^{n\times m}$ that have a complement staircase structure w.r.t.~$F$, i.e., $M_{ij}=0$ for all $i=1,\ldots,n$ and $j \leq F_i$. The staircase and complement staircase definitions are used to define what we call a causal pair of matrices.
\begin{definition}[Causal pair of matrices]\label{def:causal_pair}
	A pair of matrices $H, K^T \in \R^{n \times m}$ is said to be \emph{causal} if there exists an $(m, n)$-nondecreasing vector $F$ such that 
	\begin{equation}
		H \in  \mathcal{S}(F)\, \quad \text{and} \quad K^T\in \mathcal{S}^c(F).
	\end{equation}
\end{definition}

\begin{remark}
	The concept of causal pairs of matrices captures the full class of matrices $H$ and $K^T \in \R^{n\times m}$ such that, given $\bC:=\diag(C_1,\dots, C_m)$ with $C_i:\cH\to \cH$, the term $H \bC (K \bx^{k+1})$ in Algorithm~\ref{alg:simple_splitting_FPE_introduction} can be computed only using a single evaluation of each operator $C_i$ and simple algebraic operations.
\end{remark}

\begin{example}\label{example:causal_matrices}
	Let $F$ be a $(m,n)$-nondecreasing vector, $H\in\mathcal{S}(F)$, and $K^T\in\mathcal{S}^c(F)$. Then, $H\bC K$ is strictly lower triangular for all choices of $C_j$, which is consistent with the ordering defined by $F$. As an example, take \(m=5\), \(n=4\), and let $F=(0,2,2,5)$ that satisfies \cref{def:nondecreasing-index-set}.
	The staircase and complement staircase structures of $H$ and $K^T$, as well as the lower-triangular structure of $H\bC K$ respectively become 
	\begin{align*}
		H\,\colon\begin{pmatrix}
			0      & 0      & 0      & 0      & 0\\
			\star & \star & 0      & 0      & 0\\
			\star & \star & 0      & 0      & 0\\
			\star & \star & \star & \star & \star
		\end{pmatrix}, \quad 
		K^T\,\colon\begin{pmatrix}
			\star & \star & \star & \star & \star\\
			0      & 0      & \star & \star & \star\\
			0      & 0      & \star & \star & \star\\
			0      & 0      & 0      & 0 & 0
		\end{pmatrix},\quad \text{and} \quad
		H\bC K\,\colon\begin{pmatrix}
			0      & 0      & 0      & 0    \\
			\star & 0 & 0      & 0    \\
			\star & 0 & 0      & 0    \\
			\star & \star & \star & 0 
		\end{pmatrix}.
	\end{align*}
\end{example}

\begin{proposition}
	\label{prop:Causal_F}
	Suppose that the frugal splitting operator with minimal lifting $T$ defines an operator ordering $\hookrightarrow$. Then the ordering $\hookrightarrow$ is uniquely defined by the $(m,n)$-nondecreasing vector $F$
	given for $i\in \llbracket 1,n\rrbracket$ by
	\begin{equation}
		\label{eq: Ordering_F} 
		\begin{aligned}
			F_i := \begin{cases}
				\max \{ j\in\llbracket 1,m\rrbracket \ : \  C_j \hookrightarrow A_i\} &\text{if $C_1 \hookrightarrow A_i$},\\
				0   &\text{otherwise.}
			\end{cases}
		\end{aligned}
	\end{equation}
\end{proposition}
\begin{proof}
	By construction (see \eqref{eq:abstract_fs_block}), we note that $F_1 = 0$, and $F_n = m$ and since $A_i\hookrightarrow A_j$, $C_i\hookrightarrow C_j$ for all $i<j$, it follows from transitivity that $F_i\leq F_{i+1}$ for $i \in \llbracket 1,n-1\rrbracket$  since $A_i \hookrightarrow A_{i+1}$ which in turn is ordered before $C_j$ for each $j\in \llbracket F_{i}+1, m\rrbracket$. Any alternative $(m,n)$-nondecreasing vector, $\tilde{F}$, such that $F_{i}\neq \tilde{F_{i}}$ for some $i\in \llbracket2,n-1\rrbracket$ would have $A_i$ ordered differently with respect to $C_{\max(F_i,\tilde{F}_i)}$. 
\end{proof}

\section{Missing proofs}\label{app:missing_proofs}

\begin{proof}[Proof of Proposition~\ref{prop:causal_iff}]
	Assume that a minimal frugal splitting method parameterized by \eqref{eq:abstract_fs_block} is well-defined. By \cref{prop:Causal_F}, there exists a unique $(m, n)$-nondecreasing vector $F$. Since $A_i \hookrightarrow A_j$ and $C_i \hookrightarrow C_j$ for all $i < j$, $L$ and $Z^{(2)}$ must be strictly lower triangular. Indeed, if $L_{ij}$ is non-zero for $i\leq j$, then \eqref{eq:abstract_fs_block} yields the update $x^{+}_i = J_{\gamma_i A_i}(\gamma_i L_{ij}x^{+}_j+\dots),$ which  violates the ordering. Similarly, the ordering is violated for $u^{+}_i = C(Z^{(2)}_{ij}u^{+}_j+\dots)$ if $Z^{(2)}_{ij}\neq 0$ for $i\leq j$. This shows that $L$ and $Z^{(2)}$ are strictly lower triangular. Consider now the updates $x^{+}_i = J_{\gamma_iA_i}\left(H_{ij}C_j(\dots)+\dots\right)$ for $i\in \llbracket1, n\rrbracket$. As $A_i \hookrightarrow C_j$ for every $j\in \llbracket F_1+1, m\rrbracket$, the ordering is violated unless $H \in  \mathcal{S}(F)$. Consider instead the updates $u^{+}_i = C_i(K_{ij}x^{+}_j + \dots)$ for $i\in \llbracket 1, n\rrbracket$ and $j\in \llbracket 1, F_i\rrbracket$. As $A_i \hookrightarrow C_j$ for every $i\in \llbracket 1, n\rrbracket$ and $j\in \llbracket 1, F_i\rrbracket$, $K^T\in \mathcal{S}^c(F)$. $H$ and $K^T$ therefore constitute a causal pair of matrices. Assume, on the other hand, $L$ and $Z^{(2)}$ to be strictly lower triangular and $H\in \mathcal{S(F)}$ and $K^T\in \mathcal{S}^c(F)$ for $F$ given by \eqref{eq: Ordering_F}, given the ordering of \eqref{eq:abstract_fs_block}. Then $T$ is well-defined from the lower triangular structure of $L$, $Z^{(2)}$ and $H\bC K$.
\end{proof}

\begin{proof}[Proof of Lemma~\ref{lemma:NE_Primal_necessary}]
	Quasinonexpansiveness of the  fixed-point operator $T$, as defined in \eqref{eq:frugal_splitting_fpe_plus}, is equivalent to nonpositivity of \eqref{eq: PEP-1}. The objective function of the latter can be expressed as:
	\begin{equation*}
		\|\bz - \bz^\star - M^T(\bx - \bx^\star )\|^2 - \|\bz - \bz^\star \|^2,
	\end{equation*}
	where, e.g., $\bz$ and $\bx$ must satisfy, for some $A \in \cA_n$ and $C \in \cC_m$,
	\begin{equation*}
    (\Gamma^{-1} - L)\bx + \ba + H \bu = M \bz, \quad \bu = \bC(K \bx), \quad \text{and} \quad \ba \in \bA(\bx).
	\end{equation*}
	Similar considerations apply to $\bz^\star$ and $\bx^\star$. Via tight interpolating conditions, see \cite{RyuOperatorSplittingPEP}, these constraints can be written as in \eqref{eq: PEP-2}. Specifically, the first constraint imposes the existence of a tuple of maximal monotone operators $\ba := -(\Gamma^{-1} - L)\bx - H \bu + M \bz \in \bA(\bx)$, while the second imposes the existence of a tuple of $\frac{1}{\beta_i}$-cocoercive operators $C_1, \dots, C_m$ such that $\bu =\bC(K\bx)$. 
\end{proof}

\begin{proof}[Proof of Lemma~\ref{lemma:NE_Dual}]
	If the objective value of \eqref{eq: PEP-2} is non-positive for $(\bz^0-\bz^\star, \bx^0-\bx^\star, \bu^0-\bu^\star)\in \mathcal{H}^{n-1}\times \mathcal{H}^{n}\times \mathcal{H}^m$, then so is the objective value of \eqref{eq: SDP-PEP} for $G\in \mathbb{S}_+^{2n+m-1}$ defined as
	\begin{equation}
		G_{ij} = \langle (\bz^0-\bz^\star, \bx^0-\bx^\star, \bu^0-\bu^\star)_i, (\bz^0-\bz^\star, \bx^0-\bx^\star, \bu^0-\bu^\star)_j\rangle,
	\end{equation}
	as \eqref{eq: SDP-PEP} is simply an SDP-relaxation of \eqref{eq: PEP-2}. If instead \eqref{eq: SDP-PEP} has a positive objective value for $G\in \mathbb{S}_+^{2n+m-1}$ then for the Cholesky factorization of $G = LL^T$, let $(\bz^0-\bz^\star, \bx^0-\bx^\star, \bu^0-\bu^\star)_i := \sum_{j=1}^{2n+m-1}L_{ij}e_j$ for $i\in \llbracket 1, 2n+m-1\rrbracket$, with $\{e_i\}_{i=1}^{2n+m-1}$ an orthonormal basis of some $(2n+m-1)$-dimensional subspace of $\cH$, gives a positive objective value of \eqref{eq: PEP-2}.
\end{proof}

\begin{proof}[Proof of Lemma~\ref{lemma:Slater}]
    Consider $T$, defined by \eqref{eq:frugal_splitting_fpe_plus},
    for solving the inclusion problem with operators
    $A_i = \epsilon I$ for $i\in \llbracket1, n\rrbracket$ and $C_j = \epsilon I$ for $j\in \llbracket 1, m\rrbracket$ with $0 \leq \epsilon < \min_{i\in \{1,2,\dots m\}}\beta_i$.
    
     \emph{Step 1}: The resolvent update of \eqref{eq:frugal_splitting_fpe_plus}, initiated at $\bz\in \real^{n-1}$, produces $\bx^+\in \real^{n}$ satisfying
	\begin{equation}
		\label{eq: Slater-equality}
		(\Gamma^{-1}-L + \epsilon I + \epsilon HK)\bx^+ = M\bz.
	\end{equation} As $\Ran(M) =\one^\perp$, we note for any $\bx^+\in \real^{n}$ that there exists $\bz\in \real^{n-1}$ for which $\bx^+$ and $\bz$ solve \eqref{eq: Slater-equality} if and only if 
	\begin{equation}
		\label{eq: Slater-range}
		\one^T\left(\Gamma^{-1}-L+\epsilon I +\epsilon HK\right)\bx^+ = 0.
	\end{equation}
    Thus for $\epsilon = 0$ and $\bx^+ = \one$ there exists some $\bz\in \real^{n-1}$ satisfying equation \eqref{eq: Slater-equality}. 
    
    \emph{Step 2}: Assume instead that $\epsilon$ is strictly positive. Then the corresponding inclusion problem has the unique solution $x^\star = 0$ with a corresponding fixed point $\bz^\star = 0$.
    Consider the constraints of \eqref{eq: PEP-1}, or the relaxation in \eqref{eq: SDP-PEP}. For algorithm points consistent with this choice of operators, since $(\bz^\star, \bx^\star, \bu^\star) = (0,0,0)$, the interpolation conditions simply state that 
	\begin{equation}
        \label{eq: Slater_interpolation_conditions}
			\langle x_i, \epsilon_i x_i\rangle \geq 0\quad \text{for all} \ i\in \llbracket 1, n \rrbracket, \quad \text{and} \quad
			\langle \epsilon x_i, K_i\bx\rangle -\frac{1}{\beta_i}\|\epsilon x_i\|^2 \geq 0\quad \text{for all} \ i\in \llbracket 1, m \rrbracket
	\end{equation}
	which hold strictly if and only if $x_i \neq 0$ for all $i\in \llbracket 1, n\rrbracket$ and $\langle K_j\bx, x_j\rangle >  \frac{\epsilon}{\beta_j}\|x_j\|^2$ for all $j\in \llbracket 1, m\rrbracket$.
	
    \emph{Step 3}: Lastly, consider the function
    \begin{equation*}
        f(\bx, \varepsilon) := \one^T\left(\Gamma^{-1}-L+\epsilon I +\epsilon HK\right)\bx
    \end{equation*}
    for which zeros solve \eqref{eq: Slater-range}.
    As $f(\one, 0) = 0$, and $f$ is continuously differentiable with
    \begin{equation*}
       \nabla_\epsilon f(\one, 0) = \one^T\left(I + HK\right)\one = 2n, \quad \text{and} \quad \nabla_xf(\one, 0) = \one^T\left(\Gamma^{-1}-L\right),
    \end{equation*}
    the implicit function theorem  that
    there must exist a continuously differentiable function $s: \bx \mapsto \epsilon$ in an open neighborhood of $\bar{\bx} := \one$, such that $\bx$ solves \eqref{eq: Slater-range}, for $\varepsilon = s(\bx^+)$, and $s$ has $$\nabla s(\one) = -\frac{1}{2n}(\Gamma^{-1}-L)^T.$$
    Since this gradient is not identically zero, we have by continuity, for $\delta\in \real_{++}$ sufficiently small, that the equation \eqref{eq: Slater-range} is solved by $\bx^+ := \one + \delta\nabla s(\bx^+) > 0$ arbitrarily close to $\one$, with $0 < \epsilon := s(\bx^+) < \min_{i\in \llbracket 1, m\rrbracket} \beta_i$.
    For $\bx^+$ sufficiently close to $\one$, also $K\bx^+ > 0$, as $K\one = \one$. For such $\bx^+$ and $\epsilon$, \eqref{eq: Slater_interpolation_conditions} hold strictly.    

    Since $\bx^+$ defined in Step 3 and $\epsilon := s(\bx^+)$ solve \eqref{eq: Slater-range}, there must exist $\bz\in \real^{n-1}$ such that $\bx^+$, $\epsilon$ and $\bz$ solve \eqref{eq: Slater-equality}.
    Let us construct $G$ as the Grammian matrix of the corresponding elements $(\bz, \bx^+ , \epsilon K\bx^+)$ according to above. As $G$ solves all interpolation conditions of $\eqref{eq: SDP-PEP}$ strictly, so will the positive definite matrix $G + \epsilon_2I$ for $\epsilon_2$ positive and sufficiently small. $G+\epsilon_2I$ is a Slater point of \eqref{eq: SDP-PEP} so we have strong duality by Slater's condition.
\end{proof}

\begin{proof}[Proof of Proposition~\ref{prop:convergence_algorithm_1}]
	We reformulate \eqref{eq:algorithm_1_nonpar} as an instance of a degenerate preconditioned proximal point algorithm (dPPP) in the sense of \cite{bcln22}. Set $\hat \theta :=2\theta \in (0,2)$ and consider the equivalent algorithm with $\hat M:=\frac{1}{\sqrt{2}}M$
\begin{equation}\label{eq:degenerate_PPP1}
		\left\{\begin{aligned}
			& \bx^{k+1}=(\Gamma^{-1}-L+\bA+H\bC K)^{-1}(\hat M \bz^k),\\
			&\bz^{k+1}= \bz^k + \hat \theta \hat M^T \bx^{k+1}.
        \end{aligned}\right.
	\end{equation}
Let $\cA_L:=\bA + H\bC K + \Gamma^{-1} - L- \hat \cL$, $\hat\cL := \hat M \hat M^T$, and
	\begin{equation}
		\cA:= \begin{bmatrix}
			\cA_L & -\hat M\\
			\hat M^T & 0 
		\end{bmatrix}, \quad \text{and} \quad
	\cM:= \begin{bmatrix}
		\hat \cL^T & \hat M\\
		\hat M^T & I
	\end{bmatrix}.
	\end{equation}
Given $\bu^0:=(\bx^0, \bv^0)\in \cH^{2n-1}$, the dPPP method iterates \begin{equation}\label{eq:degenerate_PPP2}
    \bu^{k+1} =\bu^k+\hat \theta (J_{\cM^{-1}\cA}( \bu^k)-\bu^k), \quad \text{for all} \ k \in \N.
\end{equation}
Using the substitution $\bz^{k}:= \begin{bmatrix}
    \hat M^T & I
\end{bmatrix}\bu^k$ it is straightforward to see that \eqref{eq:degenerate_PPP2} coincides exactly with \eqref{eq:degenerate_PPP1}. For any $\bx,\bx'\in \cH^n$ it holds
\[\begin{aligned}
    \langle (\Gamma^{-1}-L-\hat \cL + H\bC K) (\bx) - (\Gamma^{-1}-L-\hat \cL + H\bC K) (\bx'), \bx-\bx'\rangle & = \\
    & \hspace{-10cm} = \langle (\Gamma^{-1}-L-\frac{1}{2}M M^T) (\bx-\bx'), \bx-\bx'\rangle + \langle H\bC K \bx - H\bC K \bx', \bx-\bx'\rangle \\
    & \hspace{-10cm} \geq \langle (\Gamma^{-1}-L-\frac{1}{2}MM^T) (\bx-\bx'), \bx-\bx'\rangle+ \\ & \hspace{-8cm}+ \langle \bC K \bx - \bC K \bx', (H^T-K)(\bx-\bx')\rangle + \|\bC K\bx-\bC K\bx'\|_{\diag(\beta)^{-1}}^2,
\end{aligned}\]
which is positive by positive semidefiniteness of the matrix in \eqref{eq:def_Q}, i.e., condition \eqref{ass:NE}. This implies that $H\bC K + \Gamma^{-1}-L- \hat \cL$ is monotone. The operator is also continuous and with full domain. Hence, $\cA_L$ is maximal monotone and thus also $\cA$. Since Algorithm \ref{alg:simple_splitting_FPE_introduction} can be seen as an instance of dPPP, the convergence of $\{x_i^{k+1}\}_k$ and $\{\bz^k\}_k$ follows from \cite[Corollary 2.10]{bcln22}, and the rate on the variance follows from the fact that $\Var(\bx^{k+1})$ can be upper bounded by the square fixed-point residual as in \cite[Proposition 3.5.]{bcn24}, see also \cite[Lemma 5.2.1]{chenchene_thesis}.
\end{proof}

\end{appendix}

\end{document}